\documentclass[12pt]{amsart}

\usepackage{latexsym, amssymb, amscd, amsfonts,pstricks,enumitem,amsmath}

\usepackage[all]{xypic}
\usepackage{ifthen}
\usepackage{longtable}

\renewcommand{\subsection}[1]{\vspace{3mm}\refstepcounter{subsection}\noindent{\bf \thesubsection. #1.} }

\newcommand{\np}{\vspace{3mm}\refstepcounter{subsection}\noindent{\bf \thesubsection.} }

\renewcommand{\subsubsection}[1]{\vspace{3mm}\refstepcounter{subsubsection}\noindent{\bf \thesubsubsection. #1.} }

\newcommand{\snp}{\vspace{3mm}\refstepcounter{subsubsection}\noindent{\bf \thesubsubsection.} }

\numberwithin{equation}{section}

\renewcommand{\geq}{\geqslant}
\renewcommand{\leq}{\leqslant}
\newcommand{\Osh}{{\mathcal O}}                        

\renewcommand{\H}{\mathrm{H}}                          
                            
\newcommand{\FF}{\mathbf{F}}

\newcommand{\kk}{\mathbf{k}}

\newcommand{\N}{\operatorname{N}}
\newcommand{\Vol}{\operatorname{Vol}}
\newcommand{\spec}{\operatorname{Spec}}
\newcommand{\eff}{{\mathrm{eff}}}
\newcommand{\ord}{\mathrm{ord}}

\newcommand{\KK}{\mathbf{K}}
\newcommand{\PP}{\mathbb{P}} 
\newcommand{\QQ}{\mathbb{Q}} 
\newcommand{\RR}{\mathbb{R}} 
\newcommand{\ZZ}{\mathbb{Z}} 

\newtheorem{theorem}{Theorem}[section]
\newtheorem{lemma}[theorem]{Lemma}
\newtheorem{corollary}[theorem]{Corollary}
\newtheorem{proposition}[theorem]{Proposition}
\theoremstyle{definition}

\begin{document}

\title{Diophantine approximation constants for varieties over function fields}

\author{Nathan Grieve}
\address{Department of Mathematics and Statistics,
University of New Brunswick,
Fredericton, NB, Canada}
\email{n.grieve@unb.ca}

\maketitle
 
\begin{abstract} 
By analogy with the program of McKinnon-Roth \cite{McKinnon-Roth}, we define and study approximation constants for points of a projective variety $X$ defined over $\KK$ the function field of an irreducible and non-singular in codimension $1$ projective variety defined over an algebraically closed field of characteristic zero.  In this setting, we use an effective version of Schmidt's subspace theorem, due to J.T.-Y. Wang, to give a sufficient condition for such approximation constants to be computed on a proper $\mathbf{K}$-subvariety of $X$.  We also indicate how our approximation constants are related to volume functions and Seshadri constants.
\end{abstract}

\section{Introduction}\label{intro}
The aim of this article is to study the complexity of approximating rational points of a projective variety defined over a function field of characteristic zero.  Our motivation is work of McKinnon-Roth \cite{McKinnon-Roth} and our main results, which we state in \S \ref{main:results}, show how the subspace theorem can be used to prove \emph{Roth type theorems}, by analogy with those formulated in the number field setting, see \cite[p.~515]{McKinnon-Roth}.  

Indeed, we obtain lower bounds for approximation constants of rational points.  More precisely, we show how extensions of the subspace theorem can be used to obtain lower bounds which are independent of fields of definition and which can be expressed in terms of local measures of positivity; we also give sufficient conditions for approximation constants to be computed on a proper subvariety.  As it turns out, these kinds of theorems are related to rational curves lying in projective varieties, see \S \ref{motivation} and Corollary \ref{corollary1.3}.

As we explain in \S \ref{motivation}, an important aspect to the Roth type theorems obtained in \cite{McKinnon-Roth}, in the number field setting, is a theorem of Faltings-W\"{u}stholz, \cite[Theorem 9.1]{Faltings:Wustholz}.  Understanding the role that this theorem plays in the work \cite{McKinnon-Roth} was one of the original sources of motivation for the present article.  On the other hand, one key feature to our approach here is that we use Schmidt's subspace theorem, for function fields, to derive a function field analogue of \cite[Theorem 9.1]{Faltings:Wustholz}.  We then use this result, Corollary \ref{corollary5.5}, to prove Roth type theorems in a manner similar to what is done in \cite{McKinnon-Roth}.

\subsection{Motivation}\label{motivation}
The starting point for this article is \cite[Theorem 9.1]{Faltings:Wustholz}, an interesting  
theorem of Faltings-W\"{u}stholz,  and its relation to work of McKinnon-Roth \cite{McKinnon-Roth}.    To motivate and place what we do here in its proper context let us describe the results of \cite{McKinnon-Roth} in some detail.  To this end, let $\KK$ be a number field, $\overline{\KK}$ an algebraic closure of $\KK$, $X$ an irreducible projective variety defined over $\KK$, and $x \in X(\overline{\KK})$.  The main focus of \cite{McKinnon-Roth} is the definition and study of an extended real number $\alpha_x(L)$ depending on a choice of ample line bundle $L$ on $X$ defined over $\KK$.  The intuitive idea is that the invariant $\alpha_x(L)$ provides a measure of how expensive it is to approximate $x$ by infinite sequences of distinct $\KK$-rational points of $X$.  A key insight of \cite{McKinnon-Roth} is that this arithmetic invariant is related not only to local measures of positivity for $L$ about $x$, namely the Seshadri constant $\epsilon_x(L)$, and $\beta_x(L)$ the relative asymptotic volume constant of $L$ with respect to $x$, but also to the question of existence of rational curves in $X$, passing through $x$ and defined over $\KK$.

More specifically, in \cite{McKinnon-Roth},  \cite[Theorem 9.1]{Faltings:Wustholz} was used to prove \cite[Theorem 6.2]{McKinnon-Roth} which asserts: if $g$ denotes the dimension of $X$, then either
$$\alpha_x(L) \geq \beta_x(L) \geq \frac{g}{g+1} \epsilon_x(L)$$ 
or
$$\alpha_x(L) = \alpha_x(L|_{W})$$
for some proper $\KK$-subvariety $W$ of $X$.  A consequence of this result is \cite[Theorem 6.3]{McKinnon-Roth} which states that $\alpha_x(L) \geq \frac{1}{2}\epsilon_x(L)$ with equality if and only if both $\alpha_x(L)$ and $\epsilon_x(L)$ are computed on a $\KK$-rational curve $C$ such that $C$ is unibranch at $x$, $\kappa(x) \not = \KK$, $\kappa(x) \subseteq \KK_v$, and $\epsilon_{x,C}(L|_{C}) = \epsilon_{x,X}(L)$.  (Here $\kappa(x)$ denotes the residue field of $x$ and $\KK_v$ the completion of $\KK$ with respect to $v$ a place of $\KK$.)

In light of these results
D. McKinnon has conjectured:

\noindent {\bf Conjecture} (Compare also with \cite[Conjecture 4.2]{McKinnon-Roth-Louiville}){\bf .}  Let $X$ be a smooth projective variety defined over a number field $\KK$, $\overline{\KK}$ an algebraic closure of $\KK$, $x \in X(\overline{\KK})$, and $L$ an ample line bundle on $X$ defined over $\KK$.  If $\alpha_x(L) < \infty$, then there exists a $\KK$-rational curve $C \subseteq X$ containing $x$ and also containing a \emph{sequence of best approximation to $x$}.

Our purpose here is to give content to these concepts in the setting of projective varieties defined over function fields.  

\subsection{Statement of results and outline of their proof}\label{main:results}
Our main results rely on work of Julie Wang \cite{Wang:2004} and provide an analogue of \cite[Theorem 6.2]{McKinnon-Roth} for the case of  projective varieties defined over function fields.  

To describe our results in some detail, let $\overline{\kk}$ be an algebraically closed field of characteristic zero and $Y \subseteq \PP^r_{\overline{\kk}}$ an irreducible projective variety and non-singular in codimension $1$.   Let $\KK$ denote the function field of $Y$, $\overline{\KK}$ an algebraic closure of $\KK$, let $X \subseteq \PP^n_{\KK}$ be a geometrically irreducible  subvariety, and let $L=\Osh_{\PP^n_{\KK}}(1)|_{ X}$.    

Given a prime (Weil) divisor $\mathfrak{p} \subseteq Y$  and a point $x \in X(\overline{\KK})$ we define an extended non-negative real number $\alpha_x(L) = \alpha_{x,X}(L;\mathfrak{p})=\alpha_x(L;\mathfrak{p}) \in [0,\infty]$, depending on $L$, which, roughly speaking, gives a measure of the cost of approximating $x$ by an infinite sequence of distinct $\KK$-rational points $\{y_i\} \subseteq X(\KK)$ with unbounded height and converging to $x$.
 
Our goal is two fold: on the one hand we would like to relate $\alpha_x(L;\mathfrak{p})$ to local measures of positivity of $L$ about $x$ and, on the other hand, we would like to give sufficient conditions for $\alpha_x(L;\mathfrak{p})$ to be computed on a proper $\KK$-subvariety of $X$.  This is achieved by analogy with the program of \cite{McKinnon-Roth}.   

More precisely we relate $\alpha_x(L;\mathfrak{p})$ to two invariants of $x$ with respect to $L$.  To do so, let $\FF$ be the field of definition of $x$ and $X_\FF$ the base change of $X$ with respect to the field extension $\KK \rightarrow \FF$.  Next, let $\pi : \widetilde{X} = \mathrm{Bl}_x(X) \rightarrow X_\FF$ denote the blow-up of $X_\FF$ at the closed point corresponding to $x \in X(\overline{\KK})$ and let $E$ denote the exceptional divisor of $\pi$.  If $\gamma \in \RR_{\geq 0}$, then let $L_\gamma$ denote the $\RR$-line bundle $\pi^*L_\FF-\gamma E$; here $L_\FF$ denotes the pullback of $L$ to $X_\FF$ and, in what follows, we let $L_{\gamma,\overline{\KK}}$ denote the pullback of $L_\gamma$ to $\widetilde{X}_{\overline{\KK}}$ the base change of $\widetilde{X}$ with respect to $\KK \rightarrow \overline{\KK}$.

The first invariant, the \emph{relative asymptotic volume constant of $L$ with respect to $x$}, is defined by McKinnon-Roth in \cite{McKinnon-Roth} to be:
$$\beta_x(L) = \int_0^{\gamma_{\mathrm{eff}}} \frac{\mathrm{Vol}(L_{\gamma})} {\mathrm{Vol}(L)} d\gamma; $$
here $\mathrm{Vol}(L_\gamma)$ and $\mathrm{Vol}(L)$ denote the volume of the line bundles $L_{\gamma}$ and $L$ on $\widetilde{X}$ and $X$ respectively and the real number $\gamma_{\mathrm{eff}}$ is defined by:
$$\gamma_{\eff} = \gamma_{\eff,x}(L) = \sup \{\gamma \in \RR_{\geq 0} : L_{\gamma,\overline{\KK}} \text{ is numerically equivalent to an effective divisor} \}. $$
The second invariant is the \emph{Seshadri constant of $x$ with respect to $L$}:
$$\epsilon_x(L) = \sup \{ \gamma \in \RR_{\geq 0} : L_{\gamma, \overline{\KK}} \text{ is nef}\}\text{.} $$
 
Having described briefly our main concepts, our main result, which we prove in \S \ref{proof:main:results}, reads:

\begin{theorem}\label{theorem1.1}
Let $\KK$ be the function field of an  irreducible projective variety $Y \subseteq \PP^r_{\overline{\kk}}$, defined over an algebraically closed field $\overline{\kk}$ of characteristic zero, assume that $Y$ is non-singular in codimension $1$ and fix a prime divisor $\mathfrak{p} \subseteq Y$. Fix an algebraic closure $\overline{\KK}$ of $\KK$ and suppose that $X \subseteq \PP^n_{\KK}$ is a geometrically irreducible  subvariety, that $x \in X(\overline{\KK})$, and that $L = \Osh_{\PP^n_{\KK}}(1)|_{X}$.  In this setting, either
$$\alpha_x(L;\mathfrak{p}) \geq \beta_x(L) \geq \frac{\dim X}{\dim X + 1}\epsilon_x(L) $$
or 
 $$\alpha_{x,X}(L;\mathfrak{p}) = \alpha_{x,W}(L|_{W};\mathfrak{p})  $$ for some proper subvariety $W \subsetneq X$.  
\end{theorem}

In particular, note that Theorem \ref{theorem1.1} implies that $\alpha_x(L;\mathfrak{p})$ is computed on a proper $\KK$-subvariety of $X$ provided that $\alpha_x(L;\mathfrak{p}) < \beta_x(L)$.    

By analogy with \cite{McKinnon-Roth}, Theorem \ref{theorem1.1} has the following consequence:

\begin{corollary}\label{corollary1.2}
In the setting of Theorem \ref{theorem1.1}, we have that $\alpha_x(L;\mathfrak{p}) \geq \frac{1}{2} \epsilon_x(L)$.  If equality holds then $\alpha_{x,X}(L;\mathfrak{p}) = \alpha_{x,C}(L|_{ C};\mathfrak{p})$ for some curve $C \subseteq X$ defined over $\KK$.
\end{corollary}

In the case that $\KK$ has transcendence degree $1$, Corollary \ref{corollary1.2} takes the more refined form:

\begin{corollary}\label{corollary1.3}  Assume that $\KK$ is the function field of a smooth projective curve over an algebraically closed field of characteristic zero. Let $X$ be a geometrically irreducible projective variety defined over $\KK$ and $L$ a very ample line bundle on $X$ defined over $\KK$.
If $x$ is a $\overline{\KK}$-rational point of $X$, then the inequality $\alpha_x(L) \geq \frac{1}{2}\epsilon_x(L)$ holds true.  If equality holds, then $\alpha_{x,X}(L) = \alpha_{x,B}(L|_B)$ for some rational curve $B\subseteq X$ defined over $\KK$.   
\end{corollary}

Theorem \ref{theorem1.1} and Corollary \ref{corollary1.2} are proven in \S \ref{proof:main:results}, while we prove Corollary \ref{corollary1.3} in \S \ref{9.6}.
Our techniques used to prove Theorem \ref{theorem1.1} and Corollary \ref{corollary1.2} are similar to those used to establish \cite[Theorem 6.3]{McKinnon-Roth}.   Indeed, we first define approximation constants for projective varieties defined over a field $\KK$ of characteristic zero together with a set $M_\KK$ of absolute values which satisfy the product rule.  The definition we give here extends that given in \cite{McKinnon-Roth} for the case that $\KK$ is a number field.  
We then restrict our attention to the case that $\KK$ is a function field.  In this setting, the effective version of Schmidt's subspace theorem given in \cite{Wang:2004}, and which is applicable to function fields of higher dimensional varieties, plays the role of the theorem of Faltings-W\"{u}stholz \cite[Theorem 9.1]{Faltings:Wustholz}.  More precisely, in \S \ref{5} we first give an extension of the subspace theorem obtained in \cite{Wang:2004}.  We then use this extension to obtain a function field analogue of the Faltings-W\"{u}stholz theorem.  Finally, we apply this result, in a manner similar to what is done in \cite{McKinnon-Roth}, to obtain Theorem \ref{theorem1.1} and Corollary \ref{corollary1.2}.

A key aspect to deducing Corollary \ref{corollary1.3} from Corollary \ref{corollary1.2}, is to first establish Theorem  \ref{theorem9.4} which determines the nature of approximation constants for rational points of Abelian varieties over function fields of curves.  This theorem and its proof are similar to the corresponding statement in the number field setting, see for instance \cite[Second theorem on p.~98]{Serre:Mordell-Weil-Lectures}. 

As some additional comments, 
again to place our results in their proper context, let us emphasize that in order for the results of this article to have content one encounters the question of existence of $\KK$-rational points for varieties defined over function fields.  To this end we recall the main result of \cite{Graber:Harris:Starr} which asserts that if $\KK$ is the function field of a complex curve, then every rationally connected variety defined over $\KK$ has a $\KK$-rational point.

\noindent
{\bf Acknowledgements.} 
This paper has benefited from comments and suggestions from Steven Lu, Mike Roth and Julie Wang.  I also thank Mike Roth for suggesting the problem to me.  Portions of this work were completed while I was a postdoctoral fellow at McGill University and also while I was a postdoctoral fellow at the University of New Brunswick where I was financially supported by an AARMS postdoctoral fellowship.  Finally, I thank  anonymous referees for carefully reading this work and for their comments, suggestions and corrections.

\section{Preliminaries: Absolute values, product formulas, and heights}\label{2}

In this section, to fix notation and conventions which we will require in subsequent sections, we recall some concepts and results about absolute values, product formulas, and heights.  Some standard references, from which much of our presentation is based, are \cite{Lang:Algebra}, \cite{Lang:Diophantine}, and \cite{Bombieri:Gubler}.  Throughout this section $\KK$ denotes a field of characteristic zero.  In \S \ref{2.4}--\ref{2.7} we will place further restrictions on $\KK$.  Indeed, there $\KK$ will also be a function field.

\np{}\label{2.1} {\bf Absolute values.}
By an \emph{absolute value on $\KK$} we mean a real valued function 
$$|\cdot|_{v} : \KK \rightarrow \RR$$ having the properties that:
\begin{enumerate}
\item[(a)]{$|x|_{v} \geq 0$ for all $x \in \KK$ and $|x|_{v} = 0$ if and only if $x = 0$;}
\item[(b)]{$|xy|_{v} = |x|_{v}|y|_{v}$, for all $x,y \in \KK$; }
\item[(c)]{$|x+y|_{v} \leq |x|_{v} + |y|_{v}$, for all $x,y \in \KK$.}
\end{enumerate}
We say that an absolute value $|\cdot|_{v}$ is \emph{non-archimedean} if it has the property that:  
$$|x+y|_{v} \leq \max(|x|_{v},|y|_{v} ) \text{, for all $x,y\in\KK$.}
$$
If an absolute value is not non-archimedean, then we say that it is \emph{archimedean}.  Every absolute value $|\cdot|_{v}$ defines a metric on $\KK$; the distance of two elements $x,y\in\KK$ with respect to this metric is defined to be $|x-y|_{v}$.   If $|\cdot|_{v}$ is an absolute value on $\KK$, then we let $\KK_{v}$ denote the completion of $\KK$ with respect to $|\cdot |_{v}$.  

\np{}\label{2.2}{\bf The product formula.}  Let $M_{\KK}$ denote a collection of absolute values on $\KK$.  We assume that our set $M_{\KK}$ has the property that if $x \in \KK^\times$ then $|x|_v = 1$ for almost all $|\cdot|_v \in M_{\KK}$.  We do not require $M_{\KK}$ to consist of inequivalent absolute values.  We say that \emph{$M_{\KK}$ satisfies the product formula} if for each $x \in \KK^{\times}$ we have:
\begin{equation}\label{eqn2.1}
\prod\limits_{|\cdot|_{v} \in M_{\KK}} |x|_{v} = 1 \text{.}
\end{equation}
\noindent
{\bf Remark.}  Note that the definition given above is similar to \cite[Axiom 1, p.~473]{Artin:Whaples} except that we do not require $M_{\KK}$ to consist of inequivalent absolute values.  The definition we give here is motivated by the discussion given in \cite[p.~24]{Lang:Diophantine}.

\np{}\label{2.3}{\bf Heights.}
 Let $M_\KK$ be a set of absolute values on $\KK$ which satisfies the product rule and $\PP^n_{\KK} = \mathrm{Proj} \  \KK[x_0,\dots,x_n]$.  If $y = [y_0:\cdots : y_n] \in \PP^n(\KK)$ then let 
\begin{equation}\label{eqn2.2}
H_{\Osh_{\PP^n_{\KK}}(1)}(y) = \prod\limits_{|\cdot|_{v} \in M_\KK} \max\limits_i |y_i|_{v}. 
\end{equation}
The fact that $M_\KK$ satisfies the product rule ensures that the righthand side of equation \eqref{eqn2.2} is well defined. 
The number $H_{\Osh_{\PP^n_{\KK}}(1)}(y)$ is called the \emph{multiplicative  height of $y$ with respect to $\Osh_{\PP^n_{\KK}}(1)$ and $M_\KK$} and the function 
\begin{equation}\label{eqn2.3} H_{\Osh_{\PP^n_{\KK}}(1)} : \PP^n(\KK) \rightarrow \RR 
\end{equation}
is called the \emph{multiplicative height function of $\PP^n_{\KK}$ with respect to the tautological line bundle and the set $M_\KK$}.  If $X \subseteq \PP^n_{\KK}$ is a projective variety then the multiplicative height of $x \in X(\KK)$ with respect to $L = \Osh_{\PP^n_{\KK}}(1)|_{ X}$ is defined by pulling back the function \eqref{eqn2.3} and is denoted by $H_L(x)$.

\np{}\label{2.4}{\bf Example.}  Let $\overline{\kk}$ be an algebraically closed field of characteristic zero, $Y$ an irreducible projective variety over $\overline{\kk}$ and non-singular in codimension $1$.  By a \emph{prime (Weil) divisor} of $Y$ we mean a closed integral subscheme $\mathfrak{p} \subseteq Y$ of codimension $1$.

Let $\eta$ denote the generic point of $Y$ and $\KK = \Osh_{Y,\eta}$ the field of fractions of $Y$.  If $\eta_{\mathfrak{p}}$ denotes the generic point of a prime divisor $\mathfrak{p}\subseteq Y$, then its local ring $\Osh_{Y,\eta_\mathfrak{p}} \subseteq \Osh_{Y,\eta}$ is a discrete valuation ring and we let 
\begin{equation}\label{eqn2.4} 
\ord_{\mathfrak{p}} : \KK^\times \rightarrow \ZZ
\end{equation}
denote the valuation determined by $\Osh_{Y,\eta_{\mathfrak{p}}}$.

Fix an ample line bundle $\mathcal{L}$ on $Y$.  If $\mathfrak{p} \subseteq Y$ is a prime divisor, then we let $\deg_{\mathcal{L}}(\mathfrak{p})$ denote the degree of $\mathfrak{p}$ with respect to $\mathcal{L}$, see for instance \cite[A.9.38]{Bombieri:Gubler}.  Next fix $0 < \mathbf{c} < 1$ and for each prime divisor $\mathfrak{p}\subseteq Y$, let
\begin{equation}\label{eqn2.5}
|x|_{\mathfrak{p},\KK} = \begin{cases}
\mathbf{c}^{\ord_{\mathfrak{p}}(x) \deg_{\mathcal{L}}(\mathfrak{p})} & \text{ for $x \not = 0$} \\
0 & \text{ for $x = 0$.}
\end{cases}
\end{equation}
The absolute values $|\cdot|_{\mathfrak{p},\KK}$, defined for each prime divisor $\mathfrak{p}\subseteq Y$ and depending on our fixed ample line bundle $\mathcal{L}$, are non-archimedean, proper and the set
\begin{equation}\label{eqn2.6}
M_{(Y,\mathcal{L})} = \{ |\cdot|_{\mathfrak{p},\KK} : \mathfrak{p}\subseteq Y \text{ is a prime divisor}\}
\end{equation}
is a proper set of absolute values which satisfies the product rule.

Since the set $M_{(Y,\mathcal{L})}$ satisfies the product rule we can define the multiplicative and logarithmic height functions of $\PP^n_\KK$ with respect to the tautological line bundle $\Osh_{\PP^n_\KK}(1)$.  Specifically if $y=[y_0:\dots:y_n] \in \PP^n(\KK)$, then the multiplicative height of $y$ is given by
\begin{equation}\label{eqn2.7}
H_{\Osh_{\PP^n_\KK}(1)}(y) = \prod\limits_{|\cdot|_{\mathfrak{p},\KK} \in M_{(Y,\mathcal{L})}} \max\limits_i |y_i|_{\mathfrak{p}, \KK},
\end{equation}
the logarithmic height of $y$ is given by
\begin{equation}\label{eqn2.8}
h_{\Osh_{\PP^n_{\KK}}(1)}(y) = - \sum\limits_{|\cdot|_{\mathfrak{p},\KK} \in M_{(Y,\mathcal{L})}} \min\limits_i (\ord_{\mathfrak{p}}(y_i)\deg_{\mathcal{L}}(\mathfrak{p}))
\end{equation}
and the logarithmic and multiplicative height functions are related by
\begin{equation}\label{eqn2.9}
- \log_{\mathbf{c}} H_{\Osh_{\PP^n_{\KK}}(1)}(y) = h_{\Osh_{\PP^n_\KK}(1)}(y).
\end{equation}

\np{}\label{2.5}{\bf Example.}   We continue with the situation of \S \ref{2.4}, we let $\overline{\KK}$ be an algebraic closure of the function field $\KK$ and we fix $\mathbf{F} / \KK$, $\mathbf{F} \subseteq \overline{\KK}$, a finite extension of $\KK$.  Let $\phi : Y' \rightarrow Y$ be the normalization of $Y$ in $\mathbf{F}$. As in \S \ref{2.4},  every ample line bundle $\mathcal{L}'$ on $Y'$ determines a proper set of absolute values $M_{(Y',\mathcal{L}')}$ which satisfies the product rule.  In this setting, we denote elements of $M_{(Y',\mathcal{L}')}$ by $|\cdot|_{\mathfrak{p}',\mathbf{F}}$ for $\mathfrak{p}'$ a prime divisor of $Y'$.

In particular, we can take $\mathcal{L}' = \phi^*\mathcal{L}$ for $\mathcal{L}$ an ample line bundle on $Y$.  In this case, if $\mathfrak{p}'$ is a prime divisor of $Y'$ lying over $\mathfrak{p}$ a prime divisor of $Y$, we then set
\begin{equation}\label{eqn2.10}
|x|_{\mathfrak{p}'/\mathfrak{p}} = 
|\mathrm{N}_{\mathbf{F}_{\mathfrak{p}'}/\KK_{\mathfrak{p}}}(x)|_{\mathfrak{p}, \KK}^{1/[\mathbf{F}_{\mathfrak{p}'}:\KK_{\mathfrak{p}}]} 
= |x|_{\mathfrak{p}',\mathbf{F}}^{1/[\mathbf{F}_{\mathfrak{p}'}:\KK_{\mathfrak{p}}]};
\end{equation}
here $\operatorname{N}_{\mathbf{F}_{\mathfrak{p}'}/\KK_{\mathfrak{p}}}$ denotes the field norm from $\mathbf{F}_{\mathfrak{p}'}$ to $\KK_{\mathfrak{p}}$.

As explained in \cite[\S 1.3.6]{Bombieri:Gubler}, the absolute value
\begin{equation}\label{eqn2.11}
|\cdot|_{\mathfrak{p}'/\mathfrak{p}} : \mathbf{F} \rightarrow \RR
\end{equation}
extends the absolute value
\begin{equation}\label{eqn2.12}
|\cdot|_{\mathfrak{p},\KK} : \KK \rightarrow \RR.
\end{equation}

We can also normalize the absolute values of $\mathbf{F}$ relative to $\KK$.  In particular, given a prime divisor $\mathfrak{p}'$ of $Y'$ we let $|\cdot|_{\mathfrak{p}',\KK}$ denote the absolute value
\begin{equation}\label{eqn2.13}
|x|_{\mathfrak{p}',\KK} = |x|_{\mathfrak{p}',\mathbf{F}}^{1/[\mathbf{F}:\KK]} \text{, for $x \in \mathbf{F}$,}
\end{equation}
compare with \cite[Example 1.4.13]{Bombieri:Gubler}.

\np{}\label{2.6}{\bf Height functions and field extensions.}  Since the sets $M_{(Y,\mathcal{L})}$ and $M_{(Y',\mathcal{L}')}$, defined in \S \ref{2.4} and \S \ref{2.5}, satisfy the product formula we can consider the height functions that they determine.  To compare these height functions, we first note that, as explained in \cite[Example 1.4.13]{Bombieri:Gubler}, given a prime divisor $\mathfrak{p} \subseteq Y$, the set of places of $\mathbf{F}$ lying over the place of $\KK$ determined by $\mathfrak{p}$  is in bijection with the set of prime divisors of $Y'$ lying over $\mathfrak{p}$. Given a prime divisor $\mathfrak{p}$ of $Y$ and a prime divisor $\mathfrak{p}'$ of $Y'$ we sometimes use the notation $\mathfrak{p}' | \mathfrak{p}$ to indicate that $\mathfrak{p}'$ lies above $\mathfrak{p}$.

Next, note that by \cite[Corollary 1.3.2]{Bombieri:Gubler}, given a prime divisor $\mathfrak{p}$ of $Y$, we have
\begin{equation}\label{eqn2.14}
\sum_{\mathfrak{p}' | \mathfrak{p}} [\mathbf{F}_{\mathfrak{p}'} : \KK_{\mathfrak{p}}] = [\mathbf{F}:\KK].
\end{equation}
Also, since the absolute value
$|\cdot|_{\mathfrak{p}'/\mathfrak{p}} = |\cdot|_{\mathfrak{p}',\mathbf{F}}^{1/[\mathbf{F}_{\mathfrak{p}'} :\KK_{\mathfrak{p}}] } $
extends the absolute value $|\cdot|_{\mathfrak{p}, \KK}$, it follows, using \eqref{eqn2.14} and \eqref{eqn2.10}, that if $H_{\PP^n_{\KK}(1)}(\cdot)$ denotes the height function on $\PP^n(\KK)$ determined by $M_{(Y,\mathcal{L})}$ and if $H_{\Osh_{\PP^n_{\mathbf{F}}}(1)}(\cdot)$ denotes the height function on $\PP^n(\mathbf{F})$ determined by $M_{(Y',\mathcal{L}')}$, then 
\begin{equation}\label{eqn2.15}
H_{\Osh_{\PP^n_\KK}(1)}(y) = H_{\Osh_{\PP^n_{\mathbf{F}}}(1)}(y)^{1/[\mathbf{F} : \KK]}, 
\end{equation}
for all $y = [y_0:\dots:y_n] \in \PP^n(\KK)$.

At the level of logarithmic heights, the relation \eqref{eqn2.15} implies that
\begin{equation}\label{eqn2.16}
[\mathbf{F} : \KK]  h_{\Osh_{\PP^n_{\KK}}(1)}(y) = h_{\Osh_{\PP^n_{\mathbf{F}}}(1)}(y),
\end{equation}
for all $y \in \PP^n(\KK)$.

\np{}\label{2.7}{\bf Height functions, field extensions, and projective varieties.} Considerations similar to \S \ref{2.6} apply to an arbitrary projective variety $X$ over $\KK$.  In particular, given a very ample line bundle $L$ on $X$, and defined over $\KK$, let $H_L(\cdot)$ and $h_L(\cdot)$ denote, respectively, the multiplicative and logarithmic heights obtained by pulling back $H_{\Osh_{\PP^n_{\KK}(1)}}(\cdot)$ and $h_{\Osh_{\PP^n_{\KK}}(1)}(\cdot)$ with respect to some embedding $X \hookrightarrow \PP^n_{\KK}$ afforded by $L$.  

Similarly, if $X_{\mathbf{F}}= X\times_{\spec \KK} \spec \mathbf{F}$ denotes the base change of $X$ with respect to the extension $\mathbf{F} / \KK$ and $L_{\mathbf{F}}$ the pullback of $L$ to $X_{\mathbf{F}}$, then we denote by $H_{L_{\mathbf{F}}}(\cdot)$ and $h_{L_{\mathbf{F}}}(\cdot)$, respectively, the height functions determined by pulling back $H_{\Osh_{\PP^n_{\mathbf{F}}}(1)}(\cdot)$ and $h_{\Osh_{\PP^n_{\mathbf{F}}}(1)}(\cdot)$, respectively, with respect to any embedding $X \hookrightarrow \PP^n_{\mathbf{F}}$ afforded by $L_{\mathbf{F}}$.

From this point of view, we have the relations
\begin{equation}\label{eqn1.10}
H_{L}(y) = H_{L_{\mathbf{F}}}(y)^{1/ [\mathbf{F}:\KK]}
\end{equation}
and
\begin{equation}\label{eqn1.11}
[\mathbf{F}:\KK] h_L(y) = h_{L_{\mathbf{F}}}(y),
\end{equation}
for all $y \in X(\KK)$, compare with \eqref{eqn2.15} and \eqref{eqn2.16}.

\section{Distance functions and approximation constants}\label{3}

Let $\KK$ be a field of characteristic zero and $|\cdot|$ a non-archimedean absolute value on $\KK$.  In this section we define, by analogy with \cite{McKinnon-Roth}, projective distance functions and approximation constants, with respect to $|\cdot|$, for pairs $(X,L)$ for $X$ a projective variety over $\KK$ and $L$ a very ample line bundle on $X$.  In \S \ref{4} we record some properties of these distance functions which are needed  in subsequent sections.

\noindent{\bf Projective distance functions.}   We define (normalized) distance functions for projective varieties over $\KK$ with respect to non-archimedean places of $\KK$.  The reason we include a discussion about normalizing our distance functions is so that we can state Lemma \ref{lemma3.1} below which we need later in \S \ref{4}.  In that section we also record various properties of these distance functions; these properties are needed in \S \ref{9} where we establish Corollary \ref{corollary1.3}.

\np{}\label{3.5}  Given a nontrivial absolute value $|\cdot|_{v,\KK}$ on $\KK$,  we also denote by $|\cdot|_{v,\KK}$ an extension of $|\cdot|_{v,\KK}$ to $\overline{\KK}$ an algebraic closure of $\KK$.  We fix a collection of  non-archimedean places of $\KK$ which we denote by $M_{\KK}$.  Let $\mathbf{F} / \KK$, $\mathbf{F} \subseteq \overline{\KK}$, be a finite dimensional extension and $w$ a place of $\mathbf{F}$ lying over $v$.  We let $\mathbf{F}_w$ and $\mathbf{K}_v$ denote, respectively, the completions of $\mathbf{F}$ and $\KK$ with respect to $w$ and $v$.  Finally, we write $M_\mathbf{F}$ for the set of places of $\mathbf{F}$ lying above elements of $M_\KK$.

\np{}\label{3.6}  For each $v \in M_\KK$ and each $w \in M_\mathbf{F}$, lying over $v$, we define absolute values by
\begin{equation}\label{eqn3.1}
||x||_w = |\mathrm{N}_{\mathbf{F}_w / \KK_v}(x)|_{v,\KK}
\end{equation}
and
\begin{equation}\label{eqn3.2}
|x|_{w,\KK} = |\mathrm{N}_{\mathbf{F}_w / \KK_v}(x)|_{v,\KK}^{1/[\mathbf{F} : \KK]};
\end{equation}
here $\operatorname{N}_{\mathbf{F}_w/\KK_v}$ denotes the field norm from $\mathbf{F}_w$ to $\KK_v$.  The absolute value $|\cdot|_{w,\KK}$ is a representative of $w$ and the absolute value
\begin{equation}\label{eqn3.3}
||\cdot||_w^{1/[\mathbf{F}_w : \KK_v]}
\end{equation}
is a representative of $w$ extending $|\cdot|_{v,\KK}$.  In particular, 
\begin{equation}\label{eqn3.4}
|x|_{v,\KK} = ||x||_v = ||x||_w^{1/[\mathbf{F}_w : \KK_v]},
\end{equation}
for $x \in \KK$, \cite[1.3.6, p.~6]{Bombieri:Gubler}.

\np{}\label{3.7}  We can use the absolute values defined by \eqref{eqn3.1} to define projective distance functions corresponding to places $w \in M_{\mathbf{F}}$.  When we do this, we say that this distance function is \emph{normalized relative to $\mathbf{F}$} and we denote it by $d_w(\cdot,\cdot)$ or $d_v(\cdot,\cdot)_\mathbf{F}$ for a place $v$ lying below $w$  if we wish to emphasize the fact that it is normalized relative to $\mathbf{F}$.  More specifically, given $w \in M_{\mathbf{F}}$, we fix an extension of $||\cdot||_w$ to $\overline{\KK}$ and we define 
$$d_w (\cdot, \cdot) : \PP^n(\overline{\KK}) \times \PP^n(\overline{\KK}) \rightarrow [0,1] $$
by
\begin{equation}\label{eqn3.5}
d_v(x,y)_{\mathbf{F}} = d_w(x,y) = \frac{\max_{0\leq i < j \leq n}(||x_i y_j - x_j y_i||_w )}{\max_{0\leq i \leq n}(||x_i||_w) \max_{0 \leq j \leq n}(||y_j||_w)},
\end{equation}
for $x = [x_0:\dots:x_n]$ and $y=[y_0:\dots : y_n] \in \PP^n(\overline{\KK})$ and $v \in M_\KK$ lying below $w$.

We remark:

\begin{lemma}\label{lemma3.1}
If $v \in M_\KK$ and $w \in M_\mathbf{F}$ lies over $v$ then
$$ d_v(\cdot,\cdot)_\KK^{[\mathbf{F}_w : \KK_v]} = d_v(\cdot,\cdot)_{\mathbf{F}} = d_w(\cdot,\cdot).$$
\end{lemma}
\begin{proof}
Immediate from the definitions.
\end{proof}

\np{}\label{3.8}  If $X$ is a projective variety defined over $\KK$ and $L$ a very ample line bundle on $X$, then every embedding
\begin{equation}\label{eqn3.5'}
X \hookrightarrow \PP^n_{\KK},
\end{equation}
obtained by choosing a basis of a very ample linear system with $\dim V = n+1$, determines, by pulling back the distance function defined in \eqref{eqn3.5}, a projective distance function on $X$
\begin{equation}\label{eqn3.5''}
d_v(\cdot,\cdot) = d_{|\cdot|_v}(\cdot,\cdot) : X(\overline{\KK}) \times X(\overline{\KK}) \rightarrow [0,1].
\end{equation}
Such functions also behave in the same way as Lemma \ref{lemma3.1} with respect to normalizing with respect to field extensions.

\medskip

\np{}\label{3.8'}\noindent{\bf Approximation constants.}   Let $(X,L)$ be a pair consisting of a projective variety $X$ and $L$ a very ample line bundle on $X$.  We assume that $(X,L)$ is defined over $\KK$.  Fix an embedding $X \hookrightarrow \PP^n_{\KK}$ determined by a very ample linear system $V \subseteq \H^0(X,L)$, fix a set $M_\KK$ of absolute values on $\KK$ satisfying the product rule and, as in \S \ref{2.3}, let $H_L(\cdot)$ denote the multiplicative height of $X$ with respect to $L = \Osh_{\PP^n_\KK}(1)|_X$ and our set $M_\KK$.  
Given a non-archimedean absolute value, $|\cdot|_v \in M_\KK$ let $d_{|\cdot|_v}(\cdot,\cdot)$ denote the corresponding distance function defined in \eqref{eqn3.5''}.  Here we define approximation constants and our definition extends that given in \cite[Definitions 2.8 and 2.9]{McKinnon-Roth}.

\noindent
{\bf Definition.}  Fix $x \in X(\overline{\KK})$. For every infinite sequence $\{y_i\} \subseteq X(\KK)$ of distinct points with unbounded height and $d_{|\cdot|_v}(x,y_i) \to 0$ (which we sometimes denote by $\{y_i\} \to x$) define:
\begin{equation}\label{alpha:x:seq} \alpha_x(\{y_i\},L) = \inf \{\gamma \in \RR : d_{|\cdot|_v}(x,y_i)^{\gamma} H_L(y_i) \text{ is bounded from above} \} \end{equation}
and define 
$$\alpha_{x,X}(L; |\cdot|_{v})=\alpha_x(L;|\cdot|_v) = \alpha_x(L) $$
by:
\begin{multline}\label{alpha:x}  \alpha_x(L) = \inf \{ \alpha_x(\{y_i\},L) :  \{y_i\}\subseteq X(\KK) \text{ is an infinite sequence } \\
\text{of distinct points with unbounded height and $d_{|\cdot|_v}(x,y_i) \to 0$}  \} \text{.}\end{multline}
   
The intuitive idea is that $\alpha_x(L)$ provides a measure of the cost of approximating $x \in X(\overline{\KK})$ by infinite sequences of distinct $\KK$-rational points with unbounded height and converging to $x$.
 
\noindent
{\bf Remarks.}  
\begin{enumerate}
\item[(a)]{As a matter of convention, if $\{y_i\} \subseteq X(\KK)$ is an infinite sequence of distinct points with unbounded height and not converging to $x$, then we define $\alpha_x(\{y_i\},L) = \infty$.  Similarly, if there exists no infinite sequence of distinct points $\{y_i\}\subseteq X(\KK)$ with unbounded height and converging to $x$, then we define $\alpha_x(L) = \infty$.}
\item[(b)]{In the definitions \eqref{alpha:x:seq} and \eqref{alpha:x}, the reason that we restrict our attention to infinite sequences of distinct points with unbounded height is that, in general, for instance when $\KK$ is a function field, there may exist infinite sequences of distinct points with bounded height.  On the other hand, if $\{y_i\} \subseteq X(\KK)$ is an infinite sequence of distinct points with unbounded height, then $\{y_i\}$ admits a subsequence $\{y_i'\}$ with $H_L(y_i') \to \infty$.}  
\item[(c)]{Let $\{y_i\} \subseteq X(\KK)$ be an infinite sequence of distinct points with unbounded height and $\{y_i\} \to x$.  It then follows from the definitions that if $\{y_i'\} \subseteq X(\KK)$ is a subsequence of distinct points with unbounded height then $\{y_i'\} \to x$ and $\alpha_x(\{y_i'\},L) \leq \alpha_x(\{y_i\},L)$, for all $x \in X(\overline{\KK})$.   }
\item[(d)]{If $\KK$ is a number field and $\{y_i\} \subseteq X(\KK)$ an infinite sequence of distinct points then the sequence $\{H_L(y_i)\}$ is unbounded and thus the definitions \eqref{alpha:x:seq} and \eqref{alpha:x} extend those given in \cite[Definitions 2.8 and 2.9]{McKinnon-Roth}.}
\end{enumerate}

\np{}\label{3.9}{\bf Example.}\label{PP1:eg} In the case that $\KK$ is a number field and $x \in \PP^n(\KK)$, then in \cite[Lemma 2.13]{McKinnon-Roth}, it is shown that $\alpha_x(\Osh_{\PP^n_{\KK}}(1)) = 1$.   The same is true for the case that $\KK$ is the function field of a smooth projective complex curve $C$.  To see why, as in the proof of \cite[Lemma 2.13]{McKinnon-Roth}, we have $\alpha_x(\Osh_{\PP^n_{\KK}}(1)) \geq 1$.  To see that this lower bound can be achieved, as in \cite[Lemma 2.13]{McKinnon-Roth} it suffices to treat the case $n=1$ and $x = [1:0]$. To see that $\alpha_x(\Osh_{\PP^1_{\KK}}(1)) = 1$,  let $p$ be the point of $C$ corresponding to the absolute value which we used to define $\alpha_x(\Osh_{\PP^n_{\KK}}(1))$.  Let $g$ be the genus of $C$ and let $d>2g$ be an integer.  Let $s \in \KK$ denote the global section of $\Osh_C(dp)$ with $\mathrm{div}(s)=dp$.  Then $\ord_p(s) = d$ and $\ord_q(s) = 0$ for $p \not= q$.  Since $d > 2g$, $h^0(C,\Osh_C(dp))\geq g+2$ and thus $|dp|$ is base point free so we can find a $t \in \KK$ which is a global section of $\Osh_C(dp)$ and which does not vanish at $p$.  Let $y_i=[1,s^it^{-i}]$, for $i \geq 0$.  Then $d_{|\cdot|_p}(x,y_i) \to 0$ and $H_{\Osh_{\PP^1_{\KK}}}(y_i) \to \infty$ as $i \to \infty$ and also $d_{|\cdot|_p}(x,y_i) H_{\Osh_{\PP^1_{\KK}}(1)}(y_i) = 1$ for all $i$.   

\np{}\label{3.10'}{\bf Example.}  
Let $\KK$ be the function field of a smooth projective curve over an algebraically closed field with characteristic $0$.  In \S 9 we compute $\alpha_x(L)$ for $x \in A(\overline{\KK})$, for $A$ an abelian variety defined over $\KK$ and $L$ a very ample line bundle on $A$.  Specifically, we  establish an approximation theorem similar to \cite[p.~98]{Serre:Mordell-Weil-Lectures}, proven there in the number field setting, and it follows that $\alpha_x(L) = \infty$, see Theorem \ref{theorem9.4} and Corollary \ref{corollary9.4'}.

\np{}\label{3.10}{\bf Example.}\label{curve:function:field:setting}  Let $C$ be a non-singular curve defined over $\KK$, the function field of a smooth projective curve  over an algebraically closed field with characteristic $0$ and suppose that the genus of $C$ is at least one.  If $L$ is a very ample line bundle on $C$ and $x \in C(\overline{\KK})$, then $\alpha_x(L) = \infty$ as we prove in Theorem \ref{theorem9.1'}.  To get a sense for some of the ideas involved, we consider the Abel-Jacobi map $C \rightarrow A$, here $A = \operatorname{Jac}(C)$ is the Jacobian of $C$.  Let $\Theta$ be the theta divisor of $A$ and identify $C$ with its image in $A$.  Then, in this notation, we have that $\alpha_x(\Theta^{\otimes 3}|_C) \geq \alpha_x(\Theta^{\otimes 3})$, compare with \cite[Proposition 2.14 (c)]{McKinnon-Roth}.  Now note that since $\alpha_x(\Theta^{\otimes 3}) = \infty$, see \S \ref{3.10'} or Theorem \ref{theorem9.4} and Corollary \ref{corollary9.4'},  it follows that $\alpha_x(\Theta^{\otimes 3}|_C) = \infty$ too.  Finally, it follows, from our definition of approximation constants in conjunction with  properties of height functions, that $\alpha_x(L) = \infty$ for all very ample line bundles $L$ on $C$.
The same is true for singular curves with geometric genus at least $1$, see Theorem \ref{theorem9.1'}.

\section{Properties of projective distance functions}\label{4}

In this section we record some properties of the distance functions defined in \eqref{eqn3.5} and \eqref{eqn3.5''}.   In the number field setting, similar properties were established in \cite[\S 2]{McKinnon-Roth}.  The only major difference between what we do here and what is done there is that we work with bounded sets instead of compact sets. 
We omit the proof of these properties since they are evident adaptations of the corresponding statements given in \cite[\S 2]{McKinnon-Roth}.  The main reason that we record these properties is that they are needed to  establish Lemma \ref{lemma6.1} and Theorem \ref{theorem9.4}.  Throughout this section we fix a field $\KK$ of characteristic zero, an algebraic closure $\overline{\KK}$ of $\KK$ and a place $v$ of $\KK$ which we extend to $\overline{\KK}$ and also denote by $v$.  In what follows we also fix an absolute value $|\cdot| = |\cdot|_v$ on $\overline{\KK}$ representing $v$.

\np{}\label{4.1}   Let $X$ be a projective variety over $\KK$. Besides the Zariski topology on $X(\overline{\KK})$ we have a topology which is induced by that of $\KK$ with respect to $v$.  We call this topology the \emph{$v$-topology} on $X$ and it is the topology which is induced locally by open balls with respect to closed embeddings of affine open subsets of $X$ into affine spaces and the max norm with respect to $|\cdot|$.  This topology is independent of the embeddings and the equivalence class of $|\cdot|$.  As $\KK$ need not be compact with respect to the $v$-topology, the $v$-topology on $X(\overline{\KK})$ need not be compact in general.  On the other hand, to understand $X(\overline{\KK})$ in terms of $|\cdot|$ it is useful instead to work with the concept of \emph{bounded sets} of $X$, in the sense of \cite[Definition 2.6.2]{Bombieri:Gubler} or \cite[\S 6.1]{Serre:Mordell-Weil-Lectures}.

\np{}\label{4.2}  For the sake of completeness, we include some discussion about \emph{bounded sets} and we first consider the affine case. To this end, let $U$ be an affine $\KK$-variety with coordinate ring $\KK[U]$.  We say that a subset $E \subseteq U(\overline{\KK})$ is \emph{bounded} in $U$, if for every $f \in \KK[U]$, the function $|f|$ is bounded on $E$.  If $\{f_1,\dots, f_N\}$ are generators of $\KK[U]$ as a $\KK$-algebra and if the inequality
$$ \sup_{P \in E} \max_{j=1,\dots, N} |f_j(P)| < \infty$$
holds true for a subset $E \subseteq U(\overline{\KK})$, then $E$ is bounded in $U$, \cite[Lemma 2.2.9]{Bombieri:Gubler}.  Also if $\{U_\ell\}$ is a finite open covering of $U$ and if $E$ is bounded in $U$, then there are bounded subsets $E_\ell$ of $U_\ell$ such that $E = \bigcup_\ell E_\ell$, \cite[Lemma 2.2.10]{Bombieri:Gubler}.

Next, given an arbitrary variety $X$ over $\KK$,  a subset $E \subseteq X(\overline{\KK})$ is called \emph{bounded} in $X$, if there is a finite covering $\{U_i \}_{i \in I}$ of $X$ by affine open subsets and sets $E_i$ with $E_i \subseteq U_i(\overline{\KK})$ such that $E_i$ is bounded in $U_i$ and $E = \bigcup_{i \in I} E_i$, \cite[Definition 2.6.2]{Bombieri:Gubler}.  If $E$ is bounded in $X$, then for every finite covering $\{U_i\}_{i \in I}$ of $X$ by affine open subsets, there is a subdivision 
$$ E = \bigcup_{i \in I} E_i,$$  
with   $E_i \subseteq U_i(\overline{\KK})$
such that each $E_i$ is bounded in $U_i$, \cite[Remark 2.6.3]{Bombieri:Gubler}.

Finally, as explained in \cite[Example 2.6.5]{Bombieri:Gubler}, see also \cite[\S 6]{Serre:Mordell-Weil-Lectures}, the set $\PP^n(\overline{\KK})$ is bounded in $\PP^n_{\KK}$ and one way to see this is to use the standard affine covering
$$X_i := \{ x =[x_0:\dots : x_n] \in \PP^n_{\KK} : x_i \not = 0 \},  $$
 for $i \in \{0,\dots, n \}$, of $\PP^n_\KK$ together with the decomposition 
$$E_i := \{ x \in \PP^n_{\KK} : |x_i | = \max\limits_{j=0,\dots,n} |x_j| \} $$
of $E := \PP^n(\overline{\KK})$.

One consequence of the boundedness of $\PP^n(\overline{\KK})$ is that the set of $\overline{\KK}$-rational points $X(\overline{\KK})$ for $X$ a projective variety over $\KK$ is bounded; it also  follows that if $X$ is a projective variety defined over $\KK$ and $x \in X(\overline{\KK})$ a $\overline{\KK}$-rational point of $X$, then there exists an affine open subset $U\subseteq X$ with $x \in U(\overline{\KK})$ and a subset $E \subseteq U(\overline{\KK})$ bounded in $U$ and containing $x$.  We refer to such a subset as a \emph{bounded neighbourhood of $x$} in what follows.

\np{}\label{4.3}  The key point to establishing our desired properties of the distance functions, which we defined in \S \ref{3.7} and \S \ref{3.8}, is Lemma \ref{lemma4.1} below.  To state it, let $X$ be a variety over $\spec \KK$ and $U$ an affine open subset of $X_{\mathbf{F}} = X \times_{\spec \KK} \spec \mathbf{F}$ for some finite extension $\mathbf{F} / \KK$ with $\mathbf{F} \subseteq \overline{\KK}$.  Suppose given two collections of elements $u_1,\dots, u_r$ and $u_1',\dots, u_s'$ of $\Gamma(U,\Osh_{X_{\mathbf{F}}})$ which generate the same ideal. 
 
\begin{lemma}[Compare with {\cite[Lemma 2.2]{McKinnon-Roth}}]\label{lemma4.1}  In the above setting, the functions
\begin{equation}\label{eqn4.1}
\max(|u_1(\cdot)|_v,\dots,|u_r(\cdot)|_v)
\end{equation}
and
\begin{equation}\label{eqn4.2}
\max(|u_1'(\cdot)|_v,\dots,|u_s'(\cdot)|_v)
\end{equation}
are equivalent on every subset $E \subseteq U(\overline{\KK})$ which is bounded in $U$.
\end{lemma}

\begin{proof}  
In light of the discussion given in \S \ref{4.2}, the proof of Lemma \ref{lemma4.1} is an evident adaptation of the proof of \cite[Lemma 2.2]{McKinnon-Roth}.
\end{proof}

\np{}\label{4.4}  Now let $L$ and $L'$ be two very ample line bundles on a projective variety $X$, defined over $\KK$, $V\subseteq \H^0(X,L)$ and $W \subseteq \H^0(X,L')$ two very ample linear systems, $s = \dim V -1$, $r = \dim W - 1$ and fix two embeddings $j : X \hookrightarrow \PP^s$ and $j' : X \hookrightarrow \PP^r$ obtained by choosing bases for $V$ and $W$ respectively.  We wish to compare the distance functions determined by the embeddings $j$ and $j'$.  We denote these distance functions by $d_v(\cdot,\cdot)$ and $d_v'(\cdot,\cdot)$ respectively.  The main point is Proposition \ref{proposition4.3} which shows that the functions $d_v(\cdot,\cdot)$ and $d_v'(\cdot,\cdot)$ are equivalent.  Before stating Proposition \ref{proposition4.3}, we record:

\begin{lemma}[Compare with {\cite[Lemma 2.3]{McKinnon-Roth}}]\label{lemma4.2}
Let $\mathbf{F}/\KK$ be a finite extension, $\mathbf{F} \subseteq \overline{\KK}$.  Then for every point $x \in X(\mathbf{F})$ and every rational map $f: \PP^s \dashrightarrow \PP^r$ defined at $j(x)$ and such that $f\circ j = j'$ near $x$, there is a subset $E \subseteq X(\overline{\KK}) \times X(\overline{\KK})$ bounded in $X \times X$ and containing $(x,x)$ such that $d_v(\cdot,\cdot)$ and $d_v'(\cdot,\cdot)$ are equivalent on $E$.
\end{lemma}

\begin{proof}
The proof of Lemma \ref{lemma4.2} uses Lemma \ref{lemma4.1} and is an evident adaptation of the proof of \cite[Lemma 2.3]{McKinnon-Roth}.
\end{proof}

As mentioned, the distance functions determined by distinct embeddings are equivalent:

\begin{proposition}[Compare with {\cite[Proposition 2.4]{McKinnon-Roth}}]\label{proposition4.3}
Let $d_v$ and $d_v'$ be two distance functions coming from different embeddings of $X$.  Then for all finite extensions $\mathbf{F}/\KK$, $\mathbf{F} \subseteq \overline{\KK}$, $d_v$ is equivalent to $d_v'$ on $X(\mathbf{F}) \times X(\mathbf{F})$.
\end{proposition}
\begin{proof} The proof of Proposition \ref{proposition4.3} uses Lemma \ref{lemma4.2} and, considering the discussion of \S \ref{4.2}, is an evident adaptation of the proof of \cite[Proposition 2.4]{McKinnon-Roth}.
\end{proof}

\np{}\label{4.5}   Proposition \ref{proposition4.5} below, which is useful for working with distance functions locally, is a consequence of the following useful auxiliary observation.

\begin{lemma}[Compare with {\cite[Lemma 2.5]{McKinnon-Roth}}]\label{lemma4.4}
Let $x$ be a point of $X(\overline{\KK})$ and $\mathbf{F} \subseteq \overline{\KK}$ a finite extension of $\KK$ over which $x$ is defined.  Then there exists an affine open subset $U$ of $X_{\mathbf{F}} = X \times_{\spec \KK} \spec \mathbf{F}$ containing $x$ and elements $u_1,\dots,u_r$ of $\Gamma(U,\Osh_{X_{\mathbf{F}}})$ which generate the maximal ideal of $x$ and positive real constants $c \leq C$ so that
$$ c d_v(x,y) \leq \min(1,\max(|u_1(y)|_v,\dots,|u_r(y)|_v) ) \leq C d_v(x,y),$$
for all $y \in U(\mathbf{F})$.
\end{lemma}
\begin{proof}
This is an evident adaptation of the proof of \cite[Lemma 2.5]{McKinnon-Roth} and uses Lemma \ref{lemma3.1}.
\end{proof}

Lemma \ref{lemma4.4} is needed to prove the following  useful result.

\begin{proposition}[Compare with {\cite[Lemma 2.6]{McKinnon-Roth}}]\label{proposition4.5}
Let $x$ be a point of $X(\overline{\KK})$ and $\mathbf{F} \subseteq \overline{\KK}$ a finite extension of $\KK$ over which $x$ is defined.  Let $U$ be an affine open subset of $X_{\mathbf{F}} = X \times_{\spec \KK} \spec \mathbf{F}$ containing $x$.  Let $u_1,\dots,u_r$ be elements of $\Gamma(U,\Osh_{X_{\mathbf{F}}})$ which generate the maximal ideal of $x$.  Then for every sequence of points $\{x_i \} \subseteq U(\KK)$ such that $d_v(x,x_i) \to 0$ as $i \to \infty$, the functions
\begin{equation}\label{eqn4.6}
d_v(x,\cdot) 
\end{equation}
and
\begin{equation}\label{eqn4.7}
\max(|u_1(\cdot)|_v,\dots, |u_r(\cdot)|_v)
\end{equation}
are equivalent on $\{x_i\}$.  

In particular, there exists positive constants $c \leq C$ so that for all $i \geq 0$, we have that
$$c d_v(x,x_i) \leq \max(|u_1(x_i)|_v,\dots, |u_r(x_i)|_v) \leq C d_v(x,x_i). $$
\end{proposition}
 
 \begin{proof}
This is an evident adaptation of \cite[Lemma 2.6]{McKinnon-Roth} and relies on Lemma \ref{lemma4.4}.
 \end{proof} 

\section{Wang's effective Schmidt's subspace theorem}\label{5}

In subsequent sections we study the approximation constants that we defined in \S \ref{3.8} for the case that $\KK$ is a function field.  Our approach relies on a slight extension of a theorem of Julie Wang \cite{Wang:2004} and here we describe this extension.  First we make some preliminary remarks.

\np{}\label{5.1}  Our setting is that of \S \ref{2.4}.  In particular, $\overline{\kk}$ is an algebraically closed field of characteristic zero, $Y$ is an irreducible projective variety over $\overline{\kk}$, non-singular in codimension $1$ and we have fixed an ample line bundle $\mathcal{L}$ on $Y$.  We also let $M_{(Y,\mathcal{L})}$ denote the set of absolute values of the form
\begin{equation}\label{5.1}
|\cdot|_{\mathfrak{p},\KK} : \KK \rightarrow \RR,
\end{equation}
for $\KK$ the field of fractions of $Y$ and $\mathfrak{p}$ a prime divisor of $Y$, defined in \eqref{eqn2.5}.

\np{}\label{5.2}{\bf $\mathfrak{p}$-adic metrics.}  Important to our extension of the subspace theorem is the concept of $\mathfrak{p}$-adic metrics.  Such metrics are determined by prime divisors of $Y$.  To define such metrics first let $\PP^n_{\KK} = \operatorname{Proj} \KK[x_0,\dots,x_n]$.  Every prime divisor $\mathfrak{p}$ of $Y$ determines a $\mathfrak{p}$-adic metric on the tautological line bundle $\Osh_{\PP^n_{\KK}}(1)$ given locally by:
\begin{equation}\label{eqn5.2}
||\sigma(y)||_{\mathfrak{p},\KK} = \frac{|\sigma(y)|_{\mathfrak{p},\KK}}{\max_{j,k}|a_j y_k|_{\mathfrak{p},\KK}} = \min_{j,k}\left( \frac{|\sigma(y)|_{\mathfrak{p},\KK}}{|a_j y_k|_{\mathfrak{p},\KK}} \right),
\end{equation}
for nonzero 
$$ \sigma = \sum_{j=0}^n a_j x_j \in \H^0(\PP^n,\Osh_{\PP^n_{\KK}}(1)), \text{ with } a_j \in \KK.$$

Fix an extension of $|\cdot|_{\mathfrak{p},\KK}$ to $\overline{\KK}$ and, by abuse of notation, denote this extension also by $|\cdot|_{\mathfrak{p},\KK}$.  Having fixed such an extension, we obtain a $\mathfrak{p}$-adic metric on $\Osh_{\PP^n_{\mathbf{F}}}(1)$ for all finite extensions $\mathbf{F} / \KK$, $\mathbf{F} \subseteq \overline{\KK}$.  We also denote this metric by $||\cdot||_{\mathfrak{p},\KK}$ and it is defined by \eqref{eqn5.2} for all nonzero sections
$$\sigma = \sum_{j=0}^n a_j x_j \in \H^0(\PP^n_{\mathbf{F}},\Osh_{\PP^n_{\mathbf{F}}}(1)), \text{ with } a_j \in \mathbf{F}. $$

If $X \subseteq \PP^n_{\KK}$ is a subvariety and $L = \Osh_{\PP^n_{\KK}}(1)|_X$, then we let $||\cdot||_{\mathfrak{p},\KK}$ denote the $\mathfrak{p}$-adic metric $||\cdot||_{\mathfrak{p},\KK}$ on $L$ obtained by pulling back the metric \eqref{eqn5.2}.  Similarly, given a finite extension $\mathbf{F} / \KK$, $\mathbf{F} \subseteq \overline{\KK}$, we let $||\cdot||_{\mathfrak{p},\KK}$ denote the $\mathfrak{p}$-adic metric on $L_{\mathbf{F}}$, the pull-back of $L$ to $X_{\mathbf{F}} = X \times_{\spec \KK} \spec \mathbf{F}$, obtained by a fixed extension of $|\cdot|_{\mathfrak{p},\KK}$ to an absolute value on $\overline{\KK}$.

\np{}\label{5.3} {\bf Weil functions.} We also need to make some remarks concerning Weil functions.  To do so, let 
$$\sigma = \sum_{j=0}^n a_j x_j \in \H^0(\PP^n_{\KK},\Osh_{\PP^n_{\KK}}(1)) \text{, with $a_j \in \KK$}, $$ be a nonzero section and let $\operatorname{Supp}(\sigma)$ denote the hyperplane that it determines.  The \emph{Weil function} of $\sigma$ with respect to a prime divisor $\mathfrak{p}$ of $Y$ has domain $\PP^n(\KK) \backslash \operatorname{Supp}(\sigma)(\KK)$ and is defined by
\begin{equation}\label{eqn5.3}
\lambda_{\sigma,|\cdot|_{\mathfrak{p},\KK}}(y) = ( \ord_{\mathfrak{p}}(\sigma(y)) - \min_j(\ord_{\mathfrak{p}}(y_j))- \min_j(\operatorname{ord}_{\mathfrak{p}}(a_j)))\deg_{\mathcal{L}}(\mathfrak{p}).
\end{equation}

The Weil function $\lambda_{\sigma,|\cdot|_{\mathfrak{p}},\KK}$ and the $\mathfrak{p}$-adic metric $||\cdot||_{\mathfrak{p},\KK}$ are related by
\begin{equation}\label{eqn5.4}
\lambda_{\sigma,|\cdot|_{\mathfrak{p},\KK}}(y) = \log_{\mathbf{c}} || \sigma(y)||_{\mathfrak{p},\KK},
\end{equation}
for each $y \in \PP^n(\KK)\backslash \operatorname{Supp}(\sigma)(\KK)$.

When we fix an extension of $|\cdot|_{\mathfrak{p},\KK}$ to an absolute value $|\cdot|_{\mathfrak{p},\KK} : \overline{\KK} \rightarrow \RR$ we can use the relation \eqref{eqn5.4} to consider Weil functions of nonzero sections
$$ \sigma = \sum_{j=0}^n a_j x_j \in \H^0(\PP^n_{\mathbf{F}},\Osh_{\PP^n_{\mathbf{F}}}(1)), \text{ with $a_j \in \mathbf{F}$,} $$
for $\mathbf{F} / \KK$, $\mathbf{F} \subseteq \overline{\KK}$, a finite extension.

In particular, given  a nonzero section $\sigma \in \H^0(\PP^n_{\mathbf{F}},\Osh_{\PP^n_{\mathbf{F}}}(1))$, we define its Weil function with respect to $\mathfrak{p}$ to be the function $\lambda_{\sigma,|\cdot|_{\mathfrak{p}},\KK}$ defined by
\begin{equation}\label{eqn5.5}
\lambda_{\sigma,|\cdot|_{\mathfrak{p},\KK}}(y) = \log_{\mathbf{c}}||\sigma(y)||_{\mathfrak{p},\KK},
\end{equation}
for all $y \in \PP^n(\mathbf{F})\backslash \operatorname{Supp}(\sigma)(\mathbf{F})$.

\np{}\label{5.4}{\bf The subspace theorem.}  Before we establish our extension of the subspace theorem we state the main result of \cite{Wang:2004}.  Here we state this result in a slightly more general form than considered in \cite{Wang:2004} and \cite{Ru:Wang:2012}.  Indeed, there $Y$ is assumed to be non-singular and the absolute values are considered with respect to a very ample line bundle on $Y$.  Here we assume that $Y$ is non-singular in codimension $1$ and we consider absolute values with respect to an ample line bundle $\mathcal{L}$ on $Y$.  This more general setting is important to what we do here.  

\begin{theorem}[See {\cite[p.~811]{Wang:2004}} or {\cite[Theorem 17]{Ru:Wang:2012}}]\label{theorem5.1}
Fix a finite set $S$ of prime divisors of $Y$, a collection of linear forms $\sigma_1,\dots,\sigma_q$ in $\KK[x_0,\dots,x_n]$ and let $\PP^n_{\KK} = \operatorname{Proj} \KK[x_0,\dots,x_n]$.  There exists an effectively computable finite union of proper linear subspaces  $Z \subsetneq \PP^n_\KK$ such that the following holds true:
Given $\epsilon > 0$, there exists effectively computable constants $a_\epsilon$ and $b_\epsilon$ such that for every $x \in \PP^n(\KK) \backslash Z(\KK)$ either:
\begin{enumerate}
\item{$h_{\Osh_{\PP^n_{\KK}}(1)}(x) \leq a_\epsilon$ or}
\item{$\sum\limits_{\mathfrak{p} \in S} \max\limits_J \sum\limits_{j \in J} \lambda_{\sigma_j,|\cdot|_{\mathfrak{p},\KK}}(x) \leq (n+1+\epsilon) h_{\Osh_{\PP^n_{\KK}}(1)}(x) + b_{\epsilon}$;}
here the maximum is taken over all subsets $J \subseteq \{1,\dots,q\}$ such that the $\sigma_j$, for $j\in J$, are linearly independent.
\end{enumerate}
\end{theorem}
\begin{proof}
This is implied by the main result of \cite{Wang:2004} and the remark \cite[bottom of p.~812]{Wang:2004}.  See also \cite[Theorem 17]{Ru:Wang:2012} and \cite[Remark 1]{Ru:Wang:2012}.
\end{proof}

\noindent
{\bf Remark.}  As explained in \cite[Remark 18]{Ru:Wang:2012}, the constants $a_\epsilon$ and  $b_\epsilon$ appearing in Theorem \ref{theorem5.1} depend on $\epsilon$, the degree, with respect to $\mathcal{L}$, of a  canonical class of $Y$, the sum of the degrees of the $\mathfrak{p} \in S$ with respect to $\mathcal{L}$, and the heights of the linear forms $\sigma_1,\dots,\sigma_q \in \KK[x_0,\dots,x_n]$, respect to $\mathcal{L}$, as defined by \cite[(1.5) and (1.6)]{Ru:Wang:2012}.  For a description of the union of linear subspaces $Z$ appearing in Theorem \ref{theorem5.1}, we refer to \cite[\S 3]{Wang:2004}.

\np{}\label{5.5} 
By changing the order of quantifiers slightly in Theorem \ref{theorem5.1} and using our conventions about Weil functions given in \S \ref{5.3}, especially the definition given in \eqref{eqn5.5}, we can extend Wang's subspace theorem so as to allow for linear forms having coefficients in $\overline{\KK}$.   We state this result in Theorem \ref{theorem5.2} below and I am grateful  to Julie Wang for her interest in an earlier version of this work, for telling me that such an extension should follow from her \cite[p.~811]{Wang:2004}, and for suggesting a method of proof.  Having, in \S \ref{5.2} and \S \ref{5.3}, properly defined the concepts we need the proof of Theorem \ref{theorem5.2} is standard and can be compared with \cite[Remark 7.2.3]{Bombieri:Gubler}.  

\begin{theorem}\label{theorem5.2}
Fix a finite set $S$ of prime divisors of $Y$, fix a collection of linear forms $\sigma_1,\dots,\sigma_q \in \overline{\KK}[x_0,\dots,x_n]$ and let $\PP^n_{\overline{\KK}} = \operatorname{Proj} \overline{\KK}[x_0,\dots,x_n]$.  Then given $\epsilon > 0$, there exists an effectively computable finite union of proper linear subspaces $W \subsetneq \PP^n_\KK$ and effectively computable positive constants $a_\epsilon$ and $b_\epsilon$ such that for every $x \in \PP^n(\KK) \backslash W(\KK)$ either:
\begin{enumerate}
\item{$h_{\Osh_{\PP^n_{\KK}}(1)}(x) \leq a_{\epsilon}$ or}
\item{$\sum\limits_{\mathfrak{p} \in S} \max\limits_J \sum\limits_{j \in J} \lambda_{\sigma_j,|\cdot|_{\mathfrak{p},\KK}}(x)  \leq (n+1+\epsilon)h_{\Osh_{\PP^n_{\KK}}(1)}(x) + b_\epsilon $;}
\end{enumerate}
here the maximum is taken over all subsets $J \subseteq \{1,\dots, q\}$ such that the $\sigma_j$, for $j \in J$, are linearly independent.
\end{theorem}

\noindent{\bf Remark.}  In Theorem \ref{theorem5.2}, the Weil functions are given by \eqref{eqn5.5} and they depend on our fixed choice of extension $|\cdot|_{\mathfrak{p},\KK} : \overline{\KK} \rightarrow \RR$ of the absolute value $|\cdot|_{\mathfrak{p},\KK} : \KK \rightarrow \RR$ to $\overline{\KK}$. 

\begin{proof}[Proof of Theorem \ref{theorem5.2}]
Let $\mathbf{F} / \KK$, $\mathbf{F} \subseteq \overline{\KK}$ be a finite Galois extension containing the coefficients of each of the $\sigma_j$ and let $\phi : Y' \rightarrow Y$ be the normalization of $Y$ in $\mathbf{F}$.  Let $S'$ be the set of prime divisors of $Y'$ lying over the elements of $S$.  For each $\mathfrak{p} \in S$ and each $\mathfrak{p}' \in S'$ lying over $\mathfrak{p}$ recall that the absolute value $|\cdot|_{\mathfrak{p}'/\mathfrak{p}} : \FF \rightarrow \RR$, given by \eqref{eqn2.11}, extends the absolute value $|\cdot|_{\mathfrak{p},\KK}$.

Furthermore the extension $\mathbf{F}/\KK$ is Galois.  Thus, by \cite[Corollary 1.3.5]{Bombieri:Gubler}, there exists for each $\mathfrak{p} \in S$ and each $\mathfrak{p}' \in S'$ lying over $\mathfrak{p}$ an $g_{\mathfrak{p}'/\mathfrak{p}} \in \operatorname{Gal}(\mathbf{F} /\KK)$ so that 
\begin{equation}\label{eqn5.6}
|x|_{\mathfrak{p},\KK} = |g_{\mathfrak{p}'/\mathfrak{p}}(x)|_{\mathfrak{p}'/\mathfrak{p}} \text{, for $x \in \mathbf{F}$.}
\end{equation}
On the other hand, considering \eqref{eqn2.10}, we have
\begin{equation}\label{eqn5.7}
|\cdot|_{\mathfrak{p}'/\mathfrak{p}} = |\cdot|_{\mathfrak{p}',\mathbf{F}}^{1/[\mathbf{F}_{\mathfrak{p}'}:\KK_{\mathfrak{p}}]}
\end{equation}
and it follows that
\begin{equation}\label{eqn5.8}
|x|_{\mathfrak{p},\KK} = |g_{\mathfrak{p}'/\mathfrak{p}}(x)|_{\mathfrak{p}',\mathbf{F}}^{1/[\mathbf{F}_{\mathfrak{p}'} : \KK_{\mathfrak{p}}]} \text{, for all $x \in \mathbf{F}$.}
\end{equation}

For each $\mathfrak{p} \in S$ and each $\mathfrak{p}' \in S'$ lying over $\mathfrak{p}$ let $g_{\mathfrak{p}'/\mathfrak{p}} \in \operatorname{Gal}(\mathbf{F} / \KK)$ be so that \eqref{eqn5.6} holds true and set
\begin{equation}\label{eqn5.9}
\sigma_{\mathfrak{p}',j} = g_{\mathfrak{p}'/\mathfrak{p}}(\sigma_j).
\end{equation}
Then $\sigma_{\mathfrak{p}',j}$ is the linear form in $\mathbf{F}[x_0,\dots,x_n]$ obtained by applying $g_{\mathfrak{p}'/\mathfrak{p}}$ to the coefficients of $\sigma_j$. 

Let $x \in \PP^n(\KK)$ be such that $x \not \in \operatorname{Supp}(\sigma_j)$ for $j=1,\dots,q$.  Then, considering 
 \eqref{eqn5.8}, the definition \eqref{eqn5.5} and the relation \eqref{eqn2.14}, it follows that
\begin{equation}\label{eqn5.10}
\left(\sum\limits_{\mathfrak{p} \in S} \max\limits_J \sum\limits_{j \in J} \lambda_{\sigma_j,|\cdot|_{\mathfrak{p},\KK}}(x) \right)[\mathbf{F} : \KK] = \sum\limits_{\mathfrak{p}' \in S'} \max\limits_J \sum\limits_{j \in J} \lambda_{\sigma_{\mathfrak{p}',j},|\cdot|_{\mathfrak{p}',\mathbf{F}}} (x);
\end{equation}
here the maximum in the left hand side of \eqref{eqn5.10} is taken over all $J \subseteq \{1,\dots,q\}$ so that the $\sigma_j$, $j \in J$, are linearly independent, whereas the maximum in the righthand side of equation \eqref{eqn5.10} is taken over all $J \subseteq \{1,\dots,q\}$ so that the $\sigma_{\mathfrak{p}',j}$ for $j \in J$ and fixed $\mathfrak{p}'$, are linearly independent.

The righthand side of \eqref{eqn5.10} is at most
\begin{equation}\label{eqn5.11}
\sum\limits_{\mathfrak{p}' \in S'} \max\limits_J \sum\limits_{(\mathfrak{q}',j) \in J'} \lambda_{\sigma_{\mathfrak{q}',j,|\cdot|_{\mathfrak{p}',\mathbf{F}}}}(x);
\end{equation}
 here the maximum in \eqref{eqn5.11} is taken over all subsets $J' \subseteq \{ (\mathfrak{q}',j) : \mathfrak{q}' \in S', 1 \leq j \leq q \}$ for which the $\sigma_{\mathfrak{q}',j}$ are linearly independent.
 
Considering \eqref{eqn5.11}, \eqref{eqn5.10}, and \eqref{eqn2.16}, it follows, by applying Theorem \ref{theorem5.1} over $\mathbf{F}$ with respect to the linear forms $\sigma_{\mathfrak{p}',j}$, $\mathfrak{p}' \in S'$, $j=1,\dots q$, that there exists an effectively computable union of linear subspaces $Z \subseteq \PP^n_{\mathbf{F}}$ so that for all $\epsilon > 0$ there exists effectively computable constants $a_\epsilon$ and $b_\epsilon$ such that for every $x \in \PP^n(\KK) \backslash Z(\KK)$ either
\begin{enumerate}
\item[(a)]{$h_{\Osh_{\PP^n_{\KK}}(1)}(x) \leq \frac{a_\epsilon}{[\mathbf{F}:\KK]}$ or}
\item[(b)]{$\sum\limits_{\mathfrak{p} \in S} \max\limits_J \sum\limits_{j \in J} \lambda_{\sigma_j,|\cdot|_{\mathfrak{p},\KK}} (x) \leq (n+1+\epsilon) h_{\Osh_{\PP^n_{\KK}}(1)}(x) + \frac{b_\epsilon}{[\mathbf{F}:\KK]}$,}
\end{enumerate}
where the maximum in (b) above is taken over all $J \subseteq \{1,\dots,q\}$ for which the $\sigma_j$ are linearly independent.  In particular the above hold for our given fixed $\epsilon > 0$.  To produce such a $W$ defined over $\KK$, write $Z = \bigcup_i \Lambda_i$ for linear subspaces $\Lambda_i \subseteq  \PP^n_{\mathbf{F}}$ and for each $i$ replace $\Lambda_i$ by the linear span of all solutions $x \in \Lambda_i(\KK)\bigcap \PP^n(\KK)$ to the system:
\begin{enumerate}
\item[(a')]{$h_{\Osh_{\PP^n_{\KK}}(1)}(x) > \frac{a_\epsilon}{[\mathbf{F}:\KK]}$ and}
\item[(b')]{$\sum\limits_{\mathfrak{p} \in S} \max\limits_J \sum\limits_{j \in J} \lambda_{\sigma_j, |\cdot|_{\mathfrak{p},\KK}}(x) > (n+1+\epsilon) h_{\Osh_{\PP^n_{\KK}(1)}}(x) + \frac{b_\epsilon}{[\mathbf{F}:\KK]}$.}
\end{enumerate}
The union of such linear spaces $W$ is defined over $\KK$ and the conclusion of Theorem \ref{theorem5.2} holds true for all $x \in \PP^n(\KK) \backslash W(\KK)$.
\end{proof}

\np{}\label{5.6}  We now consider consequences of Theorem \ref{theorem5.2}.  To begin with we have the following result which we state in multiplicative form.

\begin{corollary}\label{corollary5.3}  Let $\mathbf{F} / \KK$, $\mathbf{F} \subseteq \overline{\KK}$, be a finite extension, fix a finite set $S$ of prime divisors of $Y$, and fix a collection of linearly independent linear forms $\sigma_1,\dots,\sigma_q \in \mathbf{F} [x_0,\dots,x_n]$.  Then given $\epsilon > 0$, there exists a proper subvariety $Z \subsetneq \PP^n_{\KK}$ and positive constants $A_\epsilon$ and $B_\epsilon$ such that if $y \in \PP^n(\KK)$ satisfies the conditions
\begin{enumerate}
\item{$H_{\Osh_{\PP^n_{\KK}}(1)}(y) > A_{\epsilon}$; and}
\item{$\prod_{\mathfrak{p}\in S}   \prod_{j =1}^q ||\sigma_j(y)||_{\mathfrak{p},\KK} < B_{\epsilon} H_{\Osh_{\PP^n_{\KK}}(1)}(y)^{-n-1-\epsilon}$; and}
\item{$y \not \in \operatorname{Supp}(\sigma_i)$, for $i = 1,\dots, q$,}
\end{enumerate}
then $y \in Z(\KK)$.
\end{corollary}
\begin{proof}
Use the relation $h_{\Osh_{\PP^n_{\KK}}(1)}(y) = - \log_{\mathbf{c}} H_{\Osh_{\PP^n_{\KK}}(1)}(y)$ to write the conclusion of Theorem \ref{theorem5.2} in multiplicative form.
\end{proof}

Next, we give an extension of Corollary \ref{corollary5.3}.  Indeed, using Corollary \ref{corollary5.3}, we obtain a function field analogue of the Faltings-W\"{u}stholz theorem, \cite[Theorem 9.1]{Faltings:Wustholz}.  This result, which we state as Corollary \ref{corollary5.5} below, should also be compared with the discussion given in \cite[bottom of p.~1301]{Evertse:Ferretti:2002}.  We also note that in formulating this result, for the sake of simplicity, we restrict our attention to the case of a single prime divisor.  Finally, we remark that Corollary \ref{corollary5.5} below plays a key role in our approach to proving Roth type theorems as we will see in Proposition \ref{proposition6.2}.

In order to state Corollary \ref{corollary5.5}, fix a non-degenerate projective variety $X \subseteq \PP^n_\KK$, let $L = \Osh_{\PP^n_\KK}(1)|_X$, fix a finite extension $\FF$ of $\KK$, $\FF \subseteq \overline{\KK}$, and let $L_\FF$ denote the pullback of $L$ to $X_\FF = X \times_{\spec \KK} \spec \FF$ via the base change $\KK \rightarrow \FF$.  Let $H_L(\cdot) : X(\KK) \rightarrow \RR$ denote the height function determined by $L$, fix a prime divisor $\mathfrak{p}$ of $Y$ and let $||\cdot||_{\mathfrak{p}, \KK}$ be the $\mathfrak{p}$-adic metric on $L$ obtained by pulling back the metric given in \eqref{eqn5.2}.  Finally, fix an extension of the absolute value $|\cdot|_{\mathfrak{p}, \KK}$ to $\overline{\KK}$.  Then, in this way, we obtain an extension of the metric $||\cdot||_{\mathfrak{p},\KK}$ to a $\mathfrak{p}$-adic metric on $L_\FF$.

We can now state:
 
\begin{corollary}\label{corollary5.5}
In the setting just described, in particular, $X \subseteq \PP^n_\KK$ is a non-degenerate projective variety, $L = \Osh_{\PP^n_{\KK}}(1)|_X$ and $s_0,\dots,s_n \in \H^0(X,L)$ are the pull-back of the coordinate functions $x_0,\dots,x_n$, let $\sigma_1,\dots,\sigma_q \in \H^0(X_{\mathbf{F}},L_{\mathbf{F}})$ be a collection of $\mathbf{F}$-linearly independent combinations of the $s_0,\dots,s_n$.  Fix real numbers $c_1,\dots,c_q \geq 0$ with the property that $c_1+\dots+c_q > n+1$.  If $\epsilon = c_1+\dots+c_q - n - 1$, then there exists a proper subvariety $Z \subsetneq X$ and positive constants $A_\epsilon$ and $B_\epsilon$ such that the following is true:
if $y \in X(\KK)$ satisfies the conditions:
\begin{enumerate}
\item{$H_L(y) > A_\epsilon$; and}
\item{$||\sigma_i(y)||_{\mathfrak{p},\KK} < B_\epsilon H_L(y)^{-c_i} $, for $i = 1,\dots, q$; and }
\item{$y \not \in \operatorname{Supp}(\sigma_i)$, for $i=1,\dots, q$,}
\end{enumerate}
then $y \in Z(\KK)$.
\end{corollary}

\begin{proof}
Applying Corollary \ref{corollary5.3} with $S = \{ \mathfrak{p} \}$ and using the definitions of $H_L(\cdot)$ and $||\cdot||_{\mathfrak{p},\KK}$, we conclude that there exists a proper subvariety $Z \subsetneq X$ and positive constants $A_\epsilon$ and $B_\epsilon$ such that if $y \in X(\KK)$ satisfies the conditions:
\begin{itemize}
\item[(a)]{ $H_L(y) > A_\epsilon$; and }
\item[(b)]{ $\prod_{j=1}^q ||\sigma_j(y) ||_{\mathfrak{p},\KK} < B_\epsilon H_L(y)^{-n-1-\epsilon}$; and}
\item[(c)]{ $y \not \in \operatorname{Supp}(\sigma_i)$, for $i = 1,\dots, q$, } 
\end{itemize}
then $y \in Z(\KK)$.

Now suppose that $y \in X(\KK)$ satisfies the conditions that:
\begin{itemize}
\item[(a')]{$H_L(y) > A_\epsilon$; and}
\item[(b')]{ $ || \sigma_i(y)||_{\mathfrak{p}, \KK} < B_\epsilon^{\frac{1}{q}}H_L(y)^{-c_i}$, for $i=1,\dots,q$; and }
\item[(c')]{ $y \not \in \operatorname{Supp}(\sigma_i)$, for $i = 1,\dots,q$.}
\end{itemize}
We then conclude that $y$ must be contained in $Z$, since if (b') holds true, then it is also true that:
$$ \prod_{i=1}^q || \sigma_i(y)||_{\mathfrak{p}, \KK} < B_\epsilon H_L(y)^{- \sum_{i=1}^q c_i} = B_\epsilon H_L(y)^{-n-1-\epsilon}. $$
\end{proof}

\section{Computing approximation constants for varieties over function fields}\label{6}
  
Let $\overline{\kk}$ be an algebraically closed field of characteristic zero, $Y$ an irreducible projective variety over $\overline{\kk}$ which is non-singular in codimension $1$ and $\mathcal{L}$ an ample line bundle on $Y$.  Let $\KK$ be the field of fractions of $Y$ and $X \subseteq \PP^n_\KK$ a geometrically irreducible projective variety.

In this section we give sufficient conditions for approximation constants $\alpha_x(L)$, for $x \in X(\overline{\KK})$ and $L = \Osh_{\PP^n_\KK}(1)|_X$, to be computed on a proper  $\KK$-subvariety of $X$.  Our conditions are related to existence of \emph{vanishing sequences} which are \emph{Diophantine constraints}.  We define these concepts in \S \ref{6.2}.  In \S \ref{7}, especially Theorem \ref{theorem7.1}, we show how the relative asymptotic volume constants of McKinnon-Roth, \cite{McKinnon-Roth}, can be used to give sufficient conditions for existence of such vanishing sequences which are Diophantine constraints.

Throughout this section we fix a prime divisor $\mathfrak{p} \subseteq Y$.  We also fix an extension of $|\cdot|_{\mathfrak{p},\KK}$, the absolute value of $\mathfrak{p}$ with respect to $\mathcal{L}$, which we defined in \eqref{eqn2.5}, to $\overline{\KK}$ a fixed algebraic closure of $\KK$.

\np{}\label{6.1}
Since $X \subseteq \PP^n_{\KK}$, we obtain a projective distance function 
\begin{equation}\label{eqn6.1} d_{\mathfrak{p}}(\cdot,\cdot) = d_{|\cdot|_{\mathfrak{p}}}(\cdot,\cdot) : X(\overline{\KK}) \times X(\overline{\KK}) \rightarrow [0,1] \end{equation}
by pulling back the function \eqref{eqn3.5}.  The function \eqref{eqn6.1} is the projective distance function of $X$ with respect to $L = \Osh_{\PP^n_{\KK}}(1)|_{ X}$, the prime divisor $\mathfrak{p} \subseteq Y$, and the sections $s_0,\dots, s_n \in \H^0(X,L)$ obtained by pulling back the coordinate functions $x_0,\dots,x_n \in \H^0(\PP^n_{\KK},\Osh_{\PP^n_{\KK}}(1))$.
If $x \in X(\overline{\KK})$ and $\mathbf{F}$ its field of definition, then let $X_{\mathbf{F}} = X \times_{\spec \KK} \spec \mathbf{F}$ and let $L_{\mathbf{F}}$ denote the pull-back of $L$ to $X_{\mathbf{F}}$ via the base change $\spec \mathbf{F} \rightarrow \spec \mathbf{K}$.

The following lemma is used in the proof of Proposition \ref{proposition6.2}.  Its main purpose is to show how, under suitable hypothesis, the metric $||\cdot||_{\mathfrak{p},\KK}$ behaves with respect to the distance function $d_\mathfrak{p}(\cdot,\cdot)$.

\begin{lemma}\label{lemma6.1}  In the above setting, fix $x \in X(\overline{\KK})$, let $\mathbf{F}$ denote the field of definition of $x$ and suppose that a nonzero global section 
$\sigma = \sum_{j=0}^n a_j s_j \text{, with } a_j \in \mathbf{F}\text{,}$  of $L_{\mathbf{F}}$ vanishes to order at least $m$ at $x$.  In particular, locally $\sigma \in \mathfrak{m}^m_x \Osh_{X_{\mathbf{F}},x}$ the $m$th power of the maximal ideal of the local ring of $x$.
Let $\{y_i\} \subseteq X(\KK)$ be an infinite sequence of distinct points with the property that  $d_{\mathfrak{p}}(x,y_i) \to 0$ as $i \to \infty$.  Then for all $\delta > 0$ and all $i \gg 0$, depending on $\delta$, 
$$||\sigma(y_i)||_{\mathfrak{p}, \KK} \leq d_{\mathfrak{p}}(x,y_i)^{m-\delta}. $$
\end{lemma}

\begin{proof}
If $z \in X(\KK)$ has homogeneous coordinates $z=[z_0:\dots : z_n]$ then locally we know:
$$||\sigma(z)||_{\mathfrak{p}, \KK} = \min_{j,k} \left(\left|  \frac{\sigma}{a_j s_k}(z)\right|_{\mathfrak{p}, \KK} \right) $$
and locally by assumption at least one
$$\frac{\sigma}{a_j s_k} \in \mathfrak{m}^m_{x} \Osh_{X_{\mathbf{F}},x}. $$
This fact together with Proposition \ref{proposition4.5} implies that for all $i \gg 0$
$$||\sigma(y_i)||_{\mathfrak{p}, \KK} \leq \mathrm{C} d_{\mathfrak{p}}(x,y_i)^m $$ for some constant $\mathrm{C}$ independent of $i$.  We also have that 
$$d_{\mathfrak{p}}(x,y_i) \to 0 $$ as $i \to 0$.  Thus for $i \gg 0$, $d_{\mathfrak{p}}(x,y_i)$ is very small and so for all $\delta > 0$, $d_{\mathfrak{p}}(x,y_i)^{-\delta}$ will exceed $\mathrm{C}$ for all $i \gg 0$.
In particular,
$$||\sigma(y_i)||_{\mathfrak{p}, \KK} \leq \mathrm{C} d_{\mathfrak{p}}(x,y_i)^m \leq d_{\mathfrak{p}}(x,y_i)^{m-\delta} $$ for all $i \gg 0$.
\end{proof}

\np{}\label{6.2}{\bf Vanishing sequences, Diophantine constraints and computing approximation constants.}  We now introduce the concept of \emph{vanishing sequences} which are \emph{Diophantine constraints}, see \S \ref{6.2.2} and \S \ref{6.2.3} respectively.  The main motivation for these notions is that they, in conjunction with  the subspace theorem, especially Corollary \ref{corollary5.5}, allow for sufficient conditions for approximation constants to be computed on a proper subvariety, see  Proposition \ref{proposition6.2} and Theorem \ref{theorem6.3}.  We should also emphasize that these results proven here are in some sense unsatisfactory because in order for them to be of use we are faced with the issue of constructing vanishing sequences which are Diophantine constraints.  As we will see in \S \ref{7} one approach to resolving this issue is related to local positivity and especially the asymptotic volume constant in the sense of McKinnon-Roth \cite{McKinnon-Roth}. 

\snp{}\label{6.2.1}
In what follows, we fix $X \subseteq \PP^n_\KK$ a geometrically irreducible projective variety and we let $L = \Osh_{\PP^n_\KK}(1)|_X$.  We also fix $x \in X(\overline{\KK})$, we let $\mathbf{F} \subseteq \overline{\KK}$ be the field of definition of $x$, and we let $X_\FF = X \times_{\spec \KK} \spec \mathbf{F}$ denote the base change of $X$ with respect to the field extension $\mathbf{F} / \KK$.  Finally, we let $L_\FF$ denote the pullback of $L$ to $X_\FF$ via the base change map $\KK \rightarrow \FF$.

Fix an integer $m \in \ZZ_{>0}$ and let $s_0,\dots, s_N \in \H^0(X,L^{\otimes m})$ denote a basis of the $\KK$-vector space $\H^0(X,L^{\otimes m})$.  Let $\sigma_1,\dots,\sigma_q \in \H^0(X_{\FF}, L^{\otimes m}_{\FF})$ denote a collection of $\FF$-linearly independent $\FF$-linear combinations of the $s_0,\dots, s_N$.  Fix rational numbers $\gamma_1,\dots,\gamma_q \in \QQ_{>0}$ with the property that $m\gamma_j \in \ZZ$ for all $j$.

Having fixed our setting, we are now able to make two definitions.

\subsubsection{Definition}\label{6.2.2}  We say that data $(m,\gamma_\bullet,\sigma_\bullet)$ as in \S \ref{6.2.1} is a \emph{vanishing sequence for $L$ at $x$ with respect to $m$ and the rational numbers $\gamma_1,\dots,\gamma_q \in \QQ_{>0}$, and defined over $\FF$} if locally the pullback of each $\sigma_i$ is an element of $\mathfrak{m}_x^{m \gamma_i} \Osh_{X_{\mathbf{F}},x}$.

\subsubsection{Definition}\label{6.2.3}  Fix a real number $R>0$ and fix a vanishing sequence $(m,\gamma_\bullet,\sigma_\bullet)$ for $L$ at $x$ with respect to $m$ and the rational numbers $\gamma_1,\dots,\gamma_q \in \QQ_{>0}$.  We say that the  vanishing sequence $(m,\gamma_\bullet,\sigma_\bullet)$ is a \emph{Diophantine constraint for $L$ with respect to $R$ and $m$ at $x$ and defined over $\FF$} if 
there exists a proper subvariety $Z \subsetneq X$ and positive constants $A$ and $B$ so that if $y \in X(\KK)$ satisfies the conditions that
\begin{enumerate}
\item[(a)]{$H_{L^{\otimes m}}(y) > A$; and}
\item[(b)]{$||\sigma_i(y)||_{\mathfrak{p},\KK} < B H_{L^{\otimes m}}(y)^{-\gamma_iR}$, for $i = 1,\dots,q$; and}
\item[(c)]{$y \not \in \mathrm{Supp}(\sigma_i)$, for $i = 1,\dots,q$, }
\end{enumerate}
then $y \in Z(\KK)$.

\subsubsection{Example}  Suppose given a vanishing sequence $(m,\gamma_\bullet,\sigma_\bullet)$ for $L$ at $x$ with respect to $m$ and rational numbers $\gamma_1,\dots, \gamma_q \in \QQ_{>0}$ and defined over $\FF$.  Fix a real number $R > 0$ and suppose that 
$$\gamma_1 + \dots + \gamma_q > \frac{h^0(X,L^{\otimes m})}{R}.$$
Then, as implied by Corollary \ref{corollary5.5}, the data $(m,\gamma_\bullet,\sigma_\bullet)$ is a Diophantine constraint for $L$ with respect to $R$ and $m$ at $x$ and defined over $\FF$.

\snp{}\label{6.2.5}
As mentioned above, vanishing sequences and Diophantine constraints are related to approximation constants:

\begin{proposition}\label{proposition6.2} 
Let $x \in X(\overline{\KK})$ have field of definition $\FF$.  Let $R>0$ be a real number, $m \in \ZZ_{>0}$ a positive integer and suppose that there exists a vanishing sequence $(m,\gamma_\bullet,\sigma_\bullet)$ for $L$ at $x$, with respect to  $m$ and the rational numbers $\gamma_1,\dots,\gamma_q \in \QQ_{>0}$ and defined over $\FF$, which is also a Diophantine constraint with respect to $R$.   Then there exists a proper Zariski closed subset $W \subsetneq X$ defined over $\KK$, containing $x$ as a $\overline{\KK}$-point, and with the property that 
$$\alpha_{x,X}(\{y_i\},L) \geq \frac{1}{R} $$
for all infinite sequences $\{y_i\} \subseteq X(\KK) \backslash W(\KK)$ of distinct points with unbounded height.
\end{proposition}

\begin{proof}
Let $c_j = \gamma_j R$, for $j = 1,\dots, q$.  The fact that the vanishing sequence $(m,\gamma_\bullet,\sigma_\bullet)$ is a Diophantine constraint with respect to $R$ implies that there exists positive constants $A$, $B$ and a proper Zariski closed subset $W \subsetneq X$ defined over $\KK$ with the property that the collection of $y \in X(\KK)$ having the properties that 
\begin{enumerate}
\item[(a)]{$H_{L^{\otimes m}}(y) > A$; and }
\item[(b)]{$||\sigma_i(y)||_{\mathfrak{p}, \KK} < B H_{L^{\otimes m}}(y)^{-c_i} $, for $i=1,\dots,q$; and}
\item[(c)]{$y \not \in \bigcup_{i=1}^q \operatorname{Supp}(\sigma_i)$}
\end{enumerate}
 is contained in $W$.   The collection of such $y$ is also contained in $W$ adjoined with $\bigcup_{i=1}^q\operatorname{Supp}(\sigma_i)$ which, since $X$ is irreducible, is a proper Zariski closed subset of $X$.  Thus, by enlarging $W$ if necessary,  we can assume that $W$ contains $\bigcup_{i=1}^q\operatorname{Supp}(\sigma_i)(\KK)$ and $x$.
 
Suppose the proposition is false for this $W$.  Then there exists an infinite sequence $\{y_i\} \subseteq X(\KK) \backslash W(\KK)$ of distinct points with unbounded height such that
\begin{equation}\label{eqn6.2}
\alpha_{x,X}(\{y_i\},L) = \frac{1}{m}\alpha_{x,X}(\{y_i\},L^{\otimes m}) < \frac{1}{R}.
\end{equation}
Trivially \eqref{eqn6.2} implies that 
\begin{equation}\label{eqn6.2'}
\alpha_{x,X}(\{y_i\},L^{\otimes m}) < m/R
\end{equation}
and using \eqref{eqn6.2'}, in conjunction with the definition of $\alpha_{x,X}(\{y_i\},L^{\otimes m})$, it follows that 
\begin{equation}\label{eqn6.2''}
d_{\mathfrak{p}}(x,y_i)^{\frac{m}{R}-\delta'} H_{L^{\otimes m}}(y_i) \to 0, 
\end{equation}
as $i \to \infty$ for all $0 < \delta' \ll 1$.  Note also that the definition of $\alpha_{x,X}(\{y_i\},L^{\otimes m})$ implies that
\begin{equation}\label{eqn6.2'''}
d_{\mathfrak{p}}(x,y_i) \to 0, 
\end{equation}
as $i \to \infty$.

We now make the following deductions.
To begin with, using Lemma \ref{lemma6.1}, we deduce that for all $\delta > 0$ and all $j$:
\begin{equation}\label{eqn6.3}
||\sigma_{j}(y_i)||^{\frac{1}{R\gamma_j}}_{\mathfrak{p}, \KK} \leq d_{\mathfrak{p}}(x,y_i)^{\frac{m}{R} - \left(\frac{\delta}{R\gamma_j} \right)}
\end{equation}
 for all $i \gg 0$ depending on $\delta$.  (Here we use the fact that $\sigma_{j} \in \mathfrak{m}^{m \gamma_j} \Osh_{X_{\mathbf{F}},x}$ combined with the fact \eqref{eqn6.2'''} so that the hypothesis of Lemma \ref{lemma6.1} ia satisfied.)

Next choose $\delta$ so that each $\delta'_j = \frac{\delta}{R\gamma_j}$ is sufficiently small.  Then, using \eqref{eqn6.3} and \eqref{eqn6.2''} above, we deduce:
\begin{equation}\label{eqn6.4}
\frac{H_{L^{\otimes m}}(y_i) ||\sigma_{j}(y_i)||_{\mathfrak{p}, \KK}^{\frac{1}{R\gamma_j}}}{B^{\frac{1}{R\gamma_j}}}
  \leq
\frac{H_{L^{\otimes m}}(y_i)d_{\mathfrak{p}}(x,y_i)^{\frac{m}{R}-(\frac{\delta}{R\gamma_j})}}
{B^{\frac{1}{R\gamma_j}}} < 1\text{,}
\end{equation}
for all $j$ and all $i \gg 0$.

Equation \eqref{eqn6.4} has the consequence that:
$$||\sigma_{j}(y_i)||_{\mathfrak{p}, \KK} < B H_{L^{\otimes m}}(y_i)^{-R \gamma_j} $$
for all $j$ and all $i \gg 0$.  Since, by passing to a subsequence if necessary, $H_L(y_i) \to \infty$, as $i \to \infty$, it follows that $H_{L^{\otimes m}}(y_i) \to \infty$ too and so we must have that $y_i \in W$ for all $i \gg 0$.  This is a contradiction.
\end{proof} 

\snp{}\label{6.2.6}
Proposition \ref{proposition6.2} implies:

\begin{theorem}\label{theorem6.3} Let $X \subseteq \PP^n_\KK$ be a geometrically irreducible subvariety, put $L = \Osh_{\PP^n_\KK}(1)|_X$ and let $x \in X(\overline{\KK})$.
Fix a real number $R>0$ and a positive integer $m \in \ZZ_{>0}$.  If $\alpha_{x,X}(L) < \frac{1}{R}$ and if there exists a vanishing sequence $(m,\gamma_\bullet,\sigma_\bullet)$ for $L$ at $x$ with respect to $m$ and defined over $\mathbf{F} \subseteq \overline{\KK}$, the field of definition of $x$, which is also a Diophantine constraint with respect to $R$, then 
$$\alpha_{x,X}(L) = \alpha_{x,W}(L|_{ W}) $$
for some proper subvariety $W \subsetneq X$ having dimension at least $1$ containing $x$ as a $\overline{\KK}$-point.
\end{theorem}
\begin{proof}
By assumption $\alpha_{x,X}(L) < \frac{1}{R}$ and there exists a vanishing sequence $(m,\gamma_\bullet,\sigma_\bullet)$ for $L$ at $x$ with respect to $m$ which is also a Diophantine constraint with respect to $R$.  By Proposition \ref{proposition6.2}, there exists a proper subvariety $W \subsetneq X$ defined over $\KK$ and containing $x$ as a $\overline{\KK}$-point so that 
$$ \alpha_{x,X}(\{y_i\},L) \geq \frac{1}{R}$$
for all infinite sequences $\{y_i\} \subseteq X(\KK) \backslash W(\KK)$ of distinct points with unbounded height.  As a consequence if $\alpha_{x,X}(\{y_i\},L) < \frac{1}{R}$ for $\{y_i\} \subseteq X(\KK)$ an infinite sequence of distinct points with unbounded height, then almost all of the $y_i$ must lie in $W(\KK)$.   In particular $W(\KK)$ must have an infinite number of $\KK$-rational points and so $W$ has dimension at least $1$.  Since $x \in W(\overline{\KK})$ the definitions imply immediately that $\alpha_{x,X}(L) = \alpha_{x,W}(L|_{ W})$.
\end{proof}

\section{Asymptotic volume functions and vanishing  sequences}\label{7}

Let $\overline{\kk}$ be an algebraically closed field of characteristic zero, $Y\subseteq \PP^r_{\overline{\kk}}$ an irreducible projective variety, non-singular in codimension $1$, $\KK$ the field of fractions of $Y$, $X \subseteq \PP^n_\KK$ a geometrically irreducible projective variety and $L = \Osh_{\PP^n_\KK}(1)|_X$.  In this section, we relate the theory we developed in \S \ref{6} to local measures of positivity for $L$ near $x \in X(\overline{\KK})$.  Our main result is Theorem \ref{theorem7.1}, which we prove in \S \ref{7.2}, and which shows how the number $\beta_x(L)$ introduced by McKinnon-Roth can be used to construct a vanishing sequence for $L$ about $x$ which is also a Diophantine constraint.  In \S \ref{7.3} we include a short discussion about $\epsilon_x(L)$, the Seshadri constant of $L$ about $x$, and indicate how it is related to the number $\beta_x(L)$.  The reason for including this discussion is that the inequality given in \eqref{eqn7.2} below, established by McKinnon-Roth in \cite[Corollary 4.4]{McKinnon-Roth}, is needed in \S \ref{proof:main:results} where we prove the results we stated in \S \ref{motivation}.

\np{}\label{7.1}{\bf Expectations and volume functions.}  Let $x \in X(\overline{\KK})$, $\FF$ the field of definition of $x$, 
$\pi : \widetilde{X} = \operatorname{Bl}_x(X) \rightarrow X_\FF$ 
the blow-up of $X_\FF$, the base change of $X$ with respect to the field extension $\KK \rightarrow \FF$, at the closed point of $X_{\FF}$ corresponding to $x$, and let $E$ denote the exceptional divisor of $\pi$.  In what follows we let $\N^1(\widetilde{X})_\RR$ denote the real Neron-Severi space of $\RR$-Cartier divisors on $\widetilde{X}$ modulo numerical equivalence and we let $\Vol(\cdot)$ denote the volume function
$$\Vol(\cdot) : \N^1(\widetilde{X})_{\RR} \rightarrow \RR. $$
In particular if $g$ equals the dimension of $\widetilde{X}$, and if $\ell$ denotes the numerical class of an integral Cartier divisor $D$ on $\widetilde{X}$, then
$$\Vol(\ell) = \limsup_{m\to \infty} \frac{h^0(\widetilde{X},\Osh_{\widetilde{X}}(mD))}{m^g / g!},$$
see \cite[\S 2.2.C, p.~148]{Laz}.

If $\gamma \in \RR_{\geq 0}$, then let $L_\gamma$ denote the $\RR$-line bundle 
$$L_\gamma = \pi^* L_{\FF} - \gamma E$$ on $\widetilde{X}$; here $L_\FF$ denotes the pullback of $L$ to $X_\FF$.  In what follows we also let $L_{\gamma,\overline{\KK}}$ denote the pullback of $L_\gamma$ to $\widetilde{X}_{\overline{\KK}}$ the base change of $\widetilde{X}$ with respect to $\KK \rightarrow \overline{\KK}$.   In addition let $\gamma_{\eff,x}(L)$ be defined by
$$\gamma_{\eff} = \gamma_{\eff,x}(L) = \sup \{\gamma \in \RR_{\geq 0} : L_{\gamma, \overline{\KK}} \text{ is numerically equivalent to an effective divisor} \}. $$
As explained in \cite[\S 4]{McKinnon-Roth}, we have:
\begin{enumerate}
\item[(a)]{$\gamma_{{\eff}} < \infty$;}
\item[(b)]{$\Vol(L_\gamma) > 0$ for all $\gamma \in [0,\gamma_{{\eff}})$;}
\item[(c)]{$\Vol(L_\gamma) = 0$ for all $\gamma > \gamma_{\eff}$; and}
\item[(d)]{$\Vol(L_{\gamma_\eff}) = 0$.}
\end{enumerate}
In \cite[\S 4]{McKinnon-Roth} the constant $\beta_x(L)$ is defined by:
$$\beta_x(L) = \int^\infty_0 \frac{\Vol(L_\gamma)}{\Vol(L)} d \gamma = \int^{\gamma_{\eff}}_0  \frac{\Vol(L_\gamma)}{\Vol(L)} d \gamma \text{,}$$
see \cite[p.~545, Definition 4.3, and Remark p.~548]{McKinnon-Roth}.

\np{}\label{7.2}{\bf Volume functions and existence of vanishing sequences.}  We wish to show how the number $\beta_x(L)$ is related to vanishing sequences and Diophantine constraints.  Indeed, we use techniques, similar to those employed in the proof of \cite[Theorem 5.1]{McKinnon-Roth}, to prove:

\begin{theorem}\label{theorem7.1}
Let $X \subseteq \PP^n_\KK$ be a geometrically irreducible subvariety and let $L = \Osh_{\PP^n_{\KK}}(1)|_X$.  Fix a real number $R > 0$ and a $\overline{\KK}$-rational point $x \in X(\overline{\KK})$.   Let $\mathbf{F}$ denote the field of definition of $x$.  If $\beta_x(L) > \frac{1}{R}$, then there exists a positive integer $m \in \ZZ_{>0}$ and a vanishing sequence $(m,\gamma_\bullet, \sigma_\bullet)$ for $L$ at $x$ with respect to  $m$ and defined over $\mathbf{F} \subseteq \overline{\KK}$ which is also a Diophantine constraint with respect to $R$.
\end{theorem}

\begin{proof}
Let $X_{\mathbf{F}}$ denote the base change of $X$ with respect to the finite field extension $\mathbf{F} / \KK$ and let
$\pi : \widetilde{X} \rightarrow X_{\mathbf{F}} $ be the blow-up of $X$ at the closed point of $X_{\mathbf{F}}$ corresponding to $x \in X(\mathbf{F})$.  Let $E$ denote the exceptional divisor, let $L_{\mathbf{F}}$ denote the pull-back of $L$ to $X_{\mathbf{F}}$ and let $L_\gamma$ denote the $\RR$-line bundle 
$ L_\gamma = \pi^* L_{\mathbf{F}} - \gamma E$ on $\widetilde{X}$, for $\gamma \in \RR_{\geq 0}$.

Since $X$ is assumed to be geometrically irreducible, we have:
$$\beta_x(L) = \int_0^{\gamma_{\mathrm{eff}}} \frac{\operatorname{Vol}(L_\gamma)}{\operatorname{Vol}(L)} d \gamma = \int_0^{\gamma_{\mathrm{eff}}} f(\gamma) d \gamma;$$
here
$$ f(\gamma) = \frac{\operatorname{Vol}(L_\gamma)}{\operatorname{Vol}(L)}.$$

By assumption we have $\beta_x(L) > \frac{1}{R}$.  This assumption in conjunction with \cite[Lemma 5.5]{McKinnon-Roth} implies existence of a positive integer $r$ and rational numbers 
$$0 < \gamma_1 < \dots < \gamma_r < \gamma_{\mathrm{eff},x}(L)$$ so that, if we set $\gamma_{r+1} = \gamma_{\eff,x}(L)$, we have:
$$ \sum\limits_{j=1}^r c_j ( f(\gamma_j) - f(\gamma_{j+1})) > 1;$$
here $c_j = R\gamma_j$, for $j=1,\dots, r$.    

We now have, for all $\gamma \geq 0$: 
$$\lim\limits_{m \to \infty} \frac{ h^0(\widetilde{X},(L^{\otimes m})_{m \gamma})}{h^0(X, L^{\otimes m}) } = f(\gamma) $$
and it follows that by taking $m \gg 0$ we can ensure that each 
$$\frac{ h^0(\widetilde{X}, (L^{\otimes m})_{m\gamma_j})}{ h^0(X, L^{\otimes m}) }$$ is sufficiently close to $f(\gamma_j)$ so that: 
\begin{equation}\label{claim3:eqn2}
1 < \frac{1}{h^0(X, L^{\otimes m})} \left( \sum\limits_{j=1}^r c_j (h^0(\widetilde{X}, (L^{\otimes m})_{m\gamma_j}) - h^0(\widetilde{X}, (L^{\otimes m})_{m\gamma_{j+1}}) )\right)\text{.}
\end{equation}
In addition, by increasing $m$ if necessary we may assume that the $\gamma_j m$ are integers and also, by \cite[Lemma 5.4.24, p.~310]{Laz} for instance, that 
\begin{equation}\label{claim3:eqn2'}\pi_* \Osh_{\widetilde{X}}(-m\gamma_j E) = \mathcal{I}_x^{m \gamma_j},
\end{equation} 
for all $j = 1,\dots,r$.

In what follows we fix such a large integer $m$ and our goal is to construct a vanishing sequence for $L$ at $x$ with respect to $m$ which is  defined over $\mathbf{F}$ and which is a Diophantine constraint with respect to $R$.  To this end, let $V$ denote the $\mathbf{F}$-vector space $\Gamma(X_{\mathbf{F}}, L_{\mathbf{F}}^{\otimes m})$, let $N = \dim V - 1$ and $V^j = \Gamma(\widetilde{X},(L^{\otimes m}  )_{m\gamma_j})$, for $j = 1,\dots, r$.

Using \eqref{claim3:eqn2'} 
we deduce:
\begin{enumerate}
\item[(a)]{$V^j = \H^0(X_{\mathbf{F}}, \mathcal{I}_x^{m \gamma_j} \otimes L_{\mathbf{F}}^{\otimes m})$, for $j = 1,\dots, r$;}
\item[(b)]{$ V^{j+1} \subseteq V^j$, for $j = 1,\dots, r-1$; and}
\item[(c)]{each element $\sigma_j$ of $V^j$ is locally an element of $\mathfrak{m}_x^{m \gamma_j} \Osh_{X_{\mathbf{F}},x}$.}
\end{enumerate}
Let $V^0 = V$, $\ell_j = \dim V^j$, for $j=0,\dots, r$, and let $s_{r,1},\dots, s_{r,\ell_r}$ be an $\mathbf{F}$-basis for $V^r$.  We can extend this to a basis for $V^{r-1}$ which we denote by:
$s_{r,1},\dots, s_{r,\ell_r},s_{r-1,\ell_r + 1},\dots, s_{r-1,\ell_{r-1}}. $
Recursively, we can construct an $\mathbf{F}$-basis for $V^j$ extending the $\mathbf{F}$-basis for $V^{j+1}$, for $j = 1,\dots, r-1$, and we denote such a basis as:
$ s_{r,1},\dots, s_{r,\ell_r},\dots, s_{j,\ell_{j+1}+1},\dots, s_{j,\ell_j}$.  In this way, we obtain $\ell_1$ $\mathbf{F}$-linearly independent elements of the $\mathbf{F}$-vector space $V$:
$$s_{r,1},\dots, s_{r,\ell_r},\dots, s_{j,\ell_{j+1}+1},\dots, s_{j,\ell_j},\dots, s_{1,\ell_2 + 1},\dots, s_{1,\ell_1}. $$

Since the very ample line bundle $L^{\otimes m}$ is defined over $\KK$, if $s_0,\dots, s_N$ denotes a $\KK$-basis for the $\KK$-vector space $\H^0(X,L^{\otimes m})$, then each of the $\mathbf{F}$-linearly independent sections $s_{j,k}$ of $L_{\mathbf{F}}^{\otimes m}$ is an $\mathbf{F}$-linear combination of the $s_0,\dots, s_N$.

Let $\ell_{r+1} = 0$, for each $1 \leq j \leq r$ and each $\ell_{j+1} + 1 \leq k \leq \ell_j$ let the sections $s_{j,k} \in V^j$ have weight $c_{j,k} = c_j$, and let $\eta_{j,k} = \gamma_j$.  In this notation equation \eqref{claim3:eqn2} implies that:
$$ \sum\limits_{j = 1}^r \sum\limits_{k = \ell_{j+1} + 1}^{\ell_j} c_{j,k} > N+1 $$
and it follows, in light of Corollary \ref{corollary5.5}, that $(m,\eta_\bullet,\sigma_\bullet)$ with $\eta_\bullet = ( \eta_{j,\ell})$ and 
$\sigma_\bullet = (s_{j,\ell})$, for $1 \leq j \leq r$ and $\ell_{j+1} + 1 \leq \ell \leq \ell_j$,
is a vanishing sequence for $L$ with respect to $m$ at $x$ and defined over $\mathbf{F}$
which  is also a Diophantine constraint with respect to $R$.
\end{proof}

Theorem \ref{theorem7.1} has the following consequence:

\begin{corollary}\label{corollary7.2} Fix a real number $R> 0$. Continuing with the assumptions that $X$ is geometrically irreducible and $x \in X(\overline{\KK})$,  if $\beta_x(L) > \frac{1}{R}$, then there exists a proper subvariety $W \subsetneq X$ defined over $\KK$ and containing $x$  so that 
$$ \alpha_{x,X}(\{y_i\},L) \geq \frac{1}{R}$$ for all infinite sequences $\{y_i\} \subseteq X(\KK) \backslash W(\KK)$ of distinct points with unbounded height.
\end{corollary}
 
\begin{proof}
Consequence of Theorem \ref{theorem7.1} and Proposition \ref{proposition6.2}.
\end{proof}

\np{}\label{7.3}{\bf Asymptotic volume functions and their relation to Seshadri constants.} Here, in order to prepare for \S \ref{proof:main:results}, we make a few remarks about Seshadri constants and how they are related to the asymptotic relative volume constants of McKinnon-Roth \cite{McKinnon-Roth}.  To do so let
$$\epsilon_x(L)= \sup \{\gamma \in \RR_{\geq 0}: L_{\gamma, \overline{\KK}} \text{ is nef} \}$$ denote the Seshadri constant of $L$ at $x \in X(\overline{\KK})$.  
We refer to \cite[\S 3]{McKinnon-Roth} and  \cite{Laz} for more details regarding Seshadri constants.  A basic result is that, if we identify $x$ with the closed point of $X_{\overline{\KK}}$ that it determines, then 
$$\epsilon_{x,X}(L) =\inf\limits_{x \in C\subseteq X_{\overline{\KK}}} \left\{ \frac{L_{\overline{\KK}}.C}{\mathrm{mult}_x(C)} \right\}, $$
where the infimum is taken over all reduced irreducible curves $C$ passing through $x$,
see \cite[Proposition 3.2]{McKinnon-Roth} or  \cite[Proposition 5.1.5, p.~270]{Laz} for instance.

In \cite[Corollary 4.4]{McKinnon-Roth} it is shown that 
\begin{equation}\label{eqn7.2} \beta_x(L) \geq \frac{\dim X}{\dim X+1} \epsilon_x(L);\end{equation}
this inequality is important in the proof of our main results  stated in \S \ref{main:results}.

\section{Proof of main results}\label{proof:main:results}

In this section we prove the main results of this paper, namely Theorem \ref{theorem1.1} and Corollary \ref{corollary1.2}, which we stated in \S \ref{main:results}.  For convenience of the reader we restate these results as Theorem \ref{theorem8.1} and Corollary \ref{corollary8.1} below.
To prepare for these results, let $\KK$ be the function field of an irreducible projective variety $Y \subseteq \PP^r_{\overline{\kk}}$ defined over an algebraically closed field $\overline{\kk}$ of characteristic zero and non-singular in codimension $1$ and fix a prime divisor $\mathfrak{p} \subseteq Y$.   Fix an algebraic closure $\overline{\KK}$ of $\KK$ and suppose that $X \subseteq \PP^n_{\KK}$ is a geometrically irreducible subvariety.  The results we prove here show how the subspace theorem can be used to relate $\alpha_x(L)$, for $x \in X(\overline{\KK})$ and $L = \Osh_{\PP^n_{\KK}}(1)|_{ X}$, to $\beta_x(L)$.

\begin{theorem}\label{theorem8.1}  Let $\KK$ be the function field of an  irreducible projective variety $Y \subseteq \PP^r_{\overline{\kk}}$, defined over an algebraically closed field $\overline{\kk}$ of characteristic zero, assume that $Y$ is non-singular in codimension $1$ and fix a prime divisor $\mathfrak{p} \subseteq Y$. Fix an algebraic closure $\overline{\KK}$ of $\KK$ and suppose that $X \subseteq \PP^n_{\KK}$ is a geometrically irreducible  subvariety, that $x \in X(\overline{\KK})$, and that $L = \Osh_{\PP^n_{\KK}}(1)|_{X}$.  In this setting, either
$$\alpha_x(L;\mathfrak{p}) \geq \beta_x(L) \geq \frac{\dim X}{\dim X + 1}\epsilon_x(L)  $$
or 
 $$\alpha_{x,X}(L;\mathfrak{p}) = \alpha_{x,W}(L|_{W};\mathfrak{p})  $$ for some proper subvariety $W \subsetneq X$.  
\end{theorem}
\begin{proof}[Proof of Theorem \ref{theorem8.1} and Theorem \ref{theorem1.1}]  It suffices to show that if $\alpha_x(L;\mathfrak{p}) < \beta_x(L)$, then $X$ has dimension at least two and $\alpha_{x,X}(L) = \alpha_{x,W}(L|_{W})$ for some proper subvariety $W \subsetneq X$ having dimension at least $1$ and containing $x$.  To this end, if
$\alpha_x(L;\mathfrak{p}) < \beta_x(L)$, then we can choose $R> 0$ so that 
$$\alpha_x(L;\mathfrak{p}) < 1/R < \beta_x(L).$$ 
Since $\beta_x(L) > \frac{1}{R}$, Theorem \ref{theorem7.1} implies existence of a vanishing sequence $(m,\gamma_\bullet, \sigma_\bullet)$ for $L$ at $x$ with respect to  some positive integer $m$ which is also a Diophantine constraint with respect to $R$.  In addition we have $\alpha_{x,X}(L; \mathfrak{p})<\frac{1}{R}$.   The hypothesis of Theorem \ref{theorem6.3} is satisfied and its conclusion implies that $\alpha_{x,X}(L) = \alpha_{x,W}(L|_{W})$ for some proper subvariety $W \subsetneq X$ having dimension at least $1$ and containing $x$. 
\end{proof}

Theorem \ref{theorem8.1} has the following consequence:

\begin{corollary}\label{corollary8.1}  In the setting of Theorem \ref{theorem8.1}, we have that $\alpha_x(L;\mathfrak{p}) \geq \frac{1}{2} \epsilon_x(L)$.  If $\alpha_x(L;\mathfrak{p}) = \frac{1}{2} \epsilon_x(L)$,  then $\alpha_{x,X}(L;\mathfrak{p}) = \alpha_{x,C}(L|_{ C};\mathfrak{p})$ for some curve $C \subseteq X$ defined over $\KK$.
\end{corollary}

\begin{proof}
Follows from Theorem \ref{theorem8.1} using induction.  In more detail,  let $g$ denote the dimension of $X$. If $g \geq 1$, then 
$$\frac{g}{g+1}\epsilon_x(L) \geq \frac{1}{2}\epsilon_x(L). $$
Thus if $\alpha_x(L;\mathfrak{p}) \geq \frac{g}{g+1}\epsilon_x(L)$, then $\alpha_x(L;\mathfrak{p}) \geq \frac{1}{2}\epsilon_x(L)$.  If $\alpha_x(L;\mathfrak{p}) < \frac{g}{g+1}\epsilon_x(L)$, then Theorem \ref{theorem8.1} implies that $\alpha_x(L;\mathfrak{p}) = \alpha_{x,W}(L|_{W};\mathfrak{p})$ for some proper subvariety $W \subsetneq X$ and \cite[Lemma 2.17]{McKinnon-Roth} (proven for the case that $\KK$ is a number field but equally valid for the case that $\KK$ is a function field) implies that we may take $W$ to be irreducible over $\overline{\KK}$.   By induction, $\alpha_{x,W}(L|_{W};\mathfrak{p}) \geq \frac{1}{2}\epsilon_x(L|_{W})$.  On the other hand, $\alpha_x(L;\mathfrak{p}) = \alpha_{x,W}(L|_{W};\mathfrak{p})$ and $\epsilon_x(L) \leq \epsilon_{x,W}(L|_{W})$, by \cite[Proposition 3.4 (c)]{McKinnon-Roth}, and it follows that 
$$\alpha_x(L;\mathfrak{p}) = \alpha_{x,W}(L|_{W};\mathfrak{p}) \geq \frac{1}{2} \epsilon_{x}(L|_{W}) \geq \frac{1}{2}\epsilon_x(L). $$
Finally if $\alpha_x(L;\mathfrak{p}) = \frac{1}{2}\epsilon(L|_{W})$, then we conclude that $W$ is a curve defined over $\KK$.  
\end{proof}

\section{Approximation constants for Abelian varieties and curves}\label{9}

Throughout this section, we let $\KK$ be the function field of a smooth projective curve $C$ over an algebraically closed field $\overline{\kk}$ of characteristic zero.   We also fix an algebraic closure $\overline{\KK}$ of $\KK$.  Our goal here is to use the properties of the distance functions, which we recorded in \S \ref{4}, to prove an approximation theorem for rational points of an abelian variety $A$ over $\KK$.
This theorem, Theorem \ref{theorem9.4}, and its proof is very similar to what is done in the number field setting, see for example \cite[p.~98--99]{Serre:Mordell-Weil-Lectures}.  We then use Theorem \ref{theorem9.4} to study approximation constants for $\overline{\KK}$-rational points of $B$ an irreducible projective curve over $\KK$.  Specifically, in \S \ref{9.6} we prove Corollary \ref{corollary1.3} which we stated in \S \ref{main:results}.

\np{}\label{9.2} We recall a special case of Roth's theorem for $\PP^1$ from which we deduce an approximation result, Theorem \ref{theorem9.2}, applicable to projective varieties over $\KK$.  To state Roth's theorem for $\PP^1$, fix $p \in C(\overline{\kk})$ and let 
$$ d_p(\cdot,\cdot) : \PP^1(\overline{\KK}) \times \PP^1(\overline{\KK}) \rightarrow [0,1]$$
denote the projective distance  function that it determines.

\begin{theorem}[Roth's theorem for $\PP^1$]\label{theorem9.1}  Let $\KK$ be the function field of a smooth projective curve over an algebraically closed field of characteristic zero.  
Let $x \in \PP^1(\overline{\KK})$ and $\delta > 2$.  Then there is no infinite sequence $\{x_i \} \subseteq \PP^1(\KK)$ of distinct points with unbounded height so that
$$d_p(x,x_i) \to 0 \text{ and }d_p(x,x_i)H_{\Osh_{\PP^1}(1)}(x_i)^{\delta} \leq 1, $$
for $i \to \infty$.
\end{theorem}
\begin{proof}
This is implied by the main theorem of \cite{wang:1996} for example.
\end{proof}

As in \cite[\S 7.3]{Serre:Mordell-Weil-Lectures}, combined with the local description of the distance functions given by Lemma \ref{lemma4.4}, Roth's theorem for $\PP^1$ implies:

\begin{theorem}[Compare with {\cite[First theorem on p.~98]{Serre:Mordell-Weil-Lectures}}]\label{theorem9.2}  Suppose that $\KK$ is the function field of a smooth projective curve over an algebraically closed field of characteristic zero.  Let $X \subseteq \PP^n_{\KK}$ be a projective variety and $L = \Osh_{\PP^n_{\KK}}(1)|_X$.  If $x \in X(\overline{\KK})$ and $\delta > 2$, then there is no infinite sequence $\{x_i\} \subseteq X(\KK)$ of distinct points with unbounded height with 
$$ d_p(x,x_i) \to 0 \text{ and } d_p(x,x_i) H_L(x_i)^\delta \leq 1,$$
for $i \to \infty$.
\end{theorem}
\begin{proof}
As in \cite[p.~98]{Serre:Mordell-Weil-Lectures}, using the local description of the distance functions $d_p(x,\cdot)$ given by Lemma \ref{lemma4.4}, Theorem \ref{theorem9.2} follows from Theorem \ref{theorem9.1} applied to the coordinates of the $x_i$.
\end{proof}

Theorem \ref{theorem9.2} can be used to give a lower bound for the approximation constant $\alpha_x(L)$.

\begin{corollary}\label{corollary9.2'}
In the setting of Theorem \ref{theorem9.2}, $\alpha_x(L) \geq 1/2$.
\end{corollary}
\begin{proof}
Immediate considering the definition of $\alpha_x(L)$ in conjunction with the fact that the approximation constant is the reciprocal of the approximation exponent.
\end{proof}
  
\np{}\label{9.3}  Our approximation theorem for rational points of an abelian variety $A$ over $\KK$ is proved in a manner similar to what is done in the number field case, see for example \cite[\S 7.3]{Serre:Mordell-Weil-Lectures}, and relies on the weak Mordell-Weil theorem.  Specifically, we use Theorem \ref{theorem9.2}, combined with the properties of the distance functions that we stated in \S \ref{4}, to prove the following result.

\begin{theorem}[Compare with {\cite[Second theorem on p.~98]{Serre:Mordell-Weil-Lectures}}]\label{theorem9.4}   Let $\KK$ be the function field of a smooth projective curve $C$ over an algebraically closed field $\overline{\kk}$ of characteristic zero and fix $p \in C(\overline{\kk})$.
If $L$ is a very ample line bundle on an abelian variety $A$  over $\KK$, $x \in A(\overline{\KK})$ and $\delta > 0$, then there is no infinite sequence of distinct points $\{x_i \} \subseteq A(\KK)$ with unbounded height and having the property that
$$d_p(x,x_i) \to 0 \text{ and }d_p(x,x_i)H_L(x_i)^\delta \leq 1, $$
for all $i \gg 0$.
\end{theorem}

\begin{proof}
Choose an embedding $A \hookrightarrow \PP^n$ afforded by $L$, let $\delta > 0$, choose an integer $m \geq 1$ with $(m^2-1)\delta > 3$, and let $\{x_i \} \subseteq A(\KK)$ be a sequence of distinct points with unbounded height with the property that
\begin{equation}\label{eqn9.2} d_p(x,x_i) \to 0 \text{ and }d_p(x,x_i)H_L(x_i)^\delta \leq 1, 
\end{equation}
as $i \to \infty$.

The weak Mordell-Weil theorem, \cite[Theorem 10.5.14]{Bombieri:Gubler},  implies that the group $A(\KK)/mA(\KK)$ is finite.  By passing to a subsequence with unbounded height if necessary, it follows that there exists $a \in A(\KK)$ and $x_i' \in A(\KK)$ so that 
\begin{equation}\label{eqn9.3}
x_i = m x_i' + a,
\end{equation}
for all $i$.

Let $d_p'(\cdot,\cdot)$ denote the distance function obtained by using the embedding
$A \xrightarrow{\tau_a} A \hookrightarrow \PP^n; $
here $\tau_a : A \rightarrow A$ denotes translation by $a \in A(\KK)$ in the group law.  The distance functions $d'_p(\cdot,\cdot)$ and $d_p(\cdot,\cdot)$ are equivalent by Proposition \ref{proposition4.3}.  Thus \eqref{eqn9.3} together with the fact that $d_p(x,x_i) \to 0$, as $i \to \infty$, implies that
$ d_p(x-a,mx_i') \to 0,$
as $i \to \infty$.
In particular 
$ \{mx_i'\} \to x - a,$
as $i \to \infty$.

Now note that since $A(\overline{\KK})$ is a divisible group, there exists $b \in A(\overline{\KK})$ so that 
$mb = x-a.$ 
As a consequence, since $ \{mx_i'\} \to x - a$
as $i \to \infty$,
we have that
$ d_p(mb,mx_i') \to 0,$
as $i \to \infty$.

Next consider the morphism $[m]_A : A \rightarrow A$,
defined by multiplication by $m$ in the group law,
near $b$.   In particular using the fact that $[m]_A$ is \'{e}tale in conjunction with Lemma \ref{lemma4.4} and Proposition \ref{proposition4.5},
we deduce that 
\begin{equation}\label{eqn9.3'}
d_{p}(b,x_i') \to 0, 
\end{equation}
as $i \to \infty$ too.

In more detail, let $\FF$ be the field of definition of $b$.  Lemma \ref{lemma4.4} implies that there exists an affine open subset $U$ of $A_\FF = A \times_{\spec \KK} \spec \FF$ and elements $u_1,\dots,u_r$ of $\Gamma(U,\Osh_{X_\FF})$ which generate the maximal ideal of $b$ and positive constants $c\leq C$ so that the inequality
\begin{equation}\label{eqn9.4}
c d_p(b,x_i') \leq \min(1,\max(|u_1(x_i')|_p,\dots,|u_r(x_i')|_p)) \leq C d_p(b,x_i')
\end{equation}
holds true for all $x_i ' \in U(\KK)$.  Now let $\mathfrak{m}_{mb}$ denote the maximal ideal of $mb$ and $\mathfrak{m}_b$ the maximal ideal of $b$.  Then, since the morphism $[m]_A$ is \'{e}tale, we have $\mathfrak{m}_{mb}\cdot\Osh_b = \mathfrak{m}_b$, see for example \cite[Ex. III 10.3]{Hart}.  Thus, if $u_1',\dots,u_r'$ are generators for $\mathfrak{m}_{mb}$ the maximal ideal of $mb$, then their pullbacks $[m]_A^* u_1',\dots, [m]_A^* u_r'$ generate $\mathfrak{m}_b$ too.

We now consider implications of Lemma \ref{lemma4.1}.  Specifically, Lemma \ref{lemma4.1} implies that the functions
\begin{equation}\label{eqn9.4'}\max(|u_1(\cdot)|_p,\dots,|u_r(\cdot)|_p) 
\end{equation}
and
\begin{equation}\label{eqn9.4''}\max(|[m]_A^* u_1'(\cdot)|_p,\dots, |[m]_A^* u_r'(\cdot)|_p) 
\end{equation}
are equivalent on every subset $E \subseteq U(\overline{\KK})$ which is bounded in $U$.  On the other hand, since
$$\max(|[m]_A^* u_1'(x_i')|_p,\dots, |[m]_A^*u_r'(x_i')|_p) = \max(|u_1'(mx_i')|_p,\dots,|u_r'(mx_i')|_p ) \to 0, $$
 as $i \to \infty$, it follows, by the equivalence of the functions \eqref{eqn9.4'} and \eqref{eqn9.4''}, that
 \begin{equation}\label{eqn9.4'''}
 \max(|u_1(x_i')|_p,\dots,|u_r(x_i')|_p) \to \infty, 
 \end{equation}
 as $i \to \infty$ too. Combining \eqref{eqn9.4} and \eqref{eqn9.4'''} we have that $d_p(b,x_i') \to 0$ as $i \to \infty$ too.

Now, by the theory of Neron-Tate heights, see for instance \cite[\S 9.2 and \S 9.3]{Bombieri:Gubler}, the function $\log H_L$ is quadratic up to a bounded function.  In particular, since $x_i = mx_i' + a$, it follows that $\{x_i'\}$ has unbounded height and also that
\begin{equation}\label{eqn9.5}
\frac{\log H_L(x_i)}{\log H_L(x_i')} \to m^2,
\end{equation}
and thus
\begin{equation}\label{eqn9.6}
H_L(x_i')^{m^2-1} \leq H_L(x_i),
\end{equation}
for all $i \gg 0$.

Consider now the isogeny 
$\tau_a \circ [m]_A : A \rightarrow A $
and let $u_1,\dots,u_r$ be regular functions which generate the maximal ideal of $b$.  By Proposition \ref{proposition4.5}, there exists a positive constant $c$ so that
\begin{equation}\label{eqn9.7}
c d_p(b,x_i') \leq \max(|u_1(x_i')|_p,\dots,|u_r(x_i')|_p).
\end{equation}
On the other hand, since $\tau_a \circ [m]_A$ is an \'{e}tale morphism we can choose $u_1,\dots, u_r$ so that there exists regular functions $u_1',\dots,u_r'$ which generate the maximal ideal of $x$ and which have the property that $$u_j'(mx_i'+a)=(\tau_a \circ [m]_A)^*u'_j(x_i') = u_j(x_i').$$  Thus, by Proposition \ref{proposition4.5}, there exists a positive constant $C$ so that
\begin{equation}\label{eqn9.8}
\max(|u_1(x_i')|_p,\dots,|u_r(x_i')|_p ) = \max(|u_1'(mx_i'+a)|_p,\dots,|u_r'(mx_i'+a)|_p) \leq  C d_p(x,x_i).
\end{equation}
Combining \eqref{eqn9.7} and \eqref{eqn9.8} we obtain
\begin{equation}\label{eqn9.9}
c d_p(b,x_i') \leq C d_p(x,x_i)
\end{equation}
and, using \eqref{eqn9.3'}, \eqref{eqn9.9} and \eqref{eqn9.6}, it follows from \eqref{eqn9.2} that, by passing to a subsequence if necessary so that $H_L(x_i') \to \infty$,
$$d_p(b,x_i') \to 0 \text{ and } d_p(b,x_i') < H_L(x_i')^{-k} ,$$
for some $k > 2$ and all $i \gg 0$ which contradicts Theorem \ref{theorem9.2}.
\end{proof}

Theorem \ref{theorem9.4} has the following consequence which we will use in \S \ref{9.5} and \S \ref{9.6}.

\begin{corollary}\label{corollary9.4'}
In the setting of Theorem \ref{theorem9.4}, we have $\alpha_x(L) = \infty$.
\end{corollary}
\begin{proof}
Considering the definition of $\alpha_x(L)$ given in \S \ref{3.8'}, the conclusion of Corollary \ref{corollary9.4'} follows immediately from Theorem \ref{theorem9.4}.
\end{proof}

\np{}\label{9.5}  We now explain how Theorem \ref{theorem9.4} and Corollary \ref{corollary9.4'} allow for the calculation of the approximation constant $\alpha_x(L)=\alpha_x(L;p)$ for $L$ a very ample line bundle on $B$ an irreducible curve over $\KK$ and $x \in B(\overline{\KK})$.  This is the content of Theorem \ref{theorem9.1'} which is proved in a manner similar to \cite[Theorem 2.16]{McKinnon-Roth}.  

\begin{theorem}\label{theorem9.1'}  Let $\KK$ be the function field of a smooth projective curve  over an algebraically closed field of characteristic zero. If $\alpha_x(L) < \infty$ for $L$ a very ample line bundle on $B$ an irreducible curve over $\KK$ and $x \in B(\overline{\KK})$, then $B$ has geometric genus equal to zero.
\end{theorem}
\begin{proof}
Let $\phi : \widetilde{B} \rightarrow B$ be the normalization morphism.  Since the pullback of $L$ to $\widetilde{B}$ via $\phi$ is ample, \cite[Ex. III 5.7 (d)]{Hart}, and since $\alpha_x(L) = N \alpha_x(L^{\otimes N})$, for $N > 0$, without loss of generality we may assume that the pullback of $L$ to $\widetilde{B}$ is very ample.

We next note that given an infinite sequence $\{x_i \} \subseteq B(\KK)$ of distinct points with unbounded height and $\{x_i\} \to x$, we may, by passing to a subsequence with unbounded height, assume that none of the $x_i$ are the finitely many points where $\phi$ is not an isomorphism.  We may also assume that the sequence $\{\phi^{-1}(x_i)\}$ converges with respect to $d_p(\cdot,\cdot)$, the distance function on $\widetilde{B}$ determined by $\phi^* L$, to one of the points $q \in \phi^{-1}(x)$.  Indeed, if the branch corresponding to $q$ has multiplicity $m_q$, then locally $\phi$ is described by functions in the $m_q$th power of the maximal ideal of $q$.  Consequently, as can be deduced from Lemma \ref{lemma4.4}, $d_p(x,\phi(q_i))$ is equivalent to $d_p(q,q_i)^{m_q}$ as $i \to \infty$.
Conversely, it is clear that given an infinite sequence $\{q_i\} \subseteq \widetilde{B}$ of distinct points with unbounded height and $\{q_i\} \to q$ for some $q \in \phi^{-1}(x)$, we then have that $\{\phi(q_i)\} \to x$.
 
Thus, to compute $\alpha_x(L)$, it suffices to consider $\alpha_x(\{x_i\},L)$ for those infinite sequences of distinct points with  unbounded height  $\{x_i\} \subseteq B(\overline{\KK})$ which arise from infinite sequences $\{q_i \} \subseteq \widetilde{B}(\KK)$ with unbounded height and converging to some $q \in \widetilde{B}(\overline{\KK})$ with $q \in \phi^{-1}(x)$.

Now given such a sequence $\{q_i \} \to q$, we have
\begin{equation}\label{eqn5.1} H_{\phi^*L}(q_i) \sim H_L(\phi(q_i)),
\end{equation}
for all $i$. 

Since $d_p(x,\phi(q_i))$ is equivalent to $d_p(q,q_i)^{m_q}$ as $i \to \infty$, this fact combined with \eqref{eqn5.1} implies that
\begin{equation}\label{eqn9.2'}
\alpha_x(\{\phi(q_i)\},L) \sim \frac{1}{m_q} \alpha_q(\{q_i\}, \phi^*L)
\end{equation}
and so, by considering the Abel-Jacobi map of $\widetilde{B}$, compare with the discussion given in \S  \ref{3.10}, Corollary \ref{corollary9.4'} implies that the righthand side of \eqref{eqn9.2'} is infinite if $B$ has geometric genus at least $1$.
\end{proof}

\np{}\label{9.6}  Having established Theorem \ref{theorem9.1'}, we are now able to refine Corollary \ref{corollary1.2} and, in particular, establish Corollary \ref{corollary1.3}.  

\begin{theorem}\label{theorem9.1''} Assume that $\KK$ is the function field of a smooth projective curve over an algebraically closed field of characteristic zero.  Let $X$ be a geometrically irreducible projective variety defined over $\KK$ and $L$ a very ample line bundle on $X$ defined over $\KK$.
If $x$ is a $\overline{\KK}$-rational point of $X$, then the inequality $\alpha_x(L) \geq \frac{1}{2}\epsilon_x(L)$ holds true.  If equality holds, then $\alpha_{x,X}(L) = \alpha_{x,B}(L|_B)$ for some rational curve $B\subseteq X$ defined over $\KK$.   
\end{theorem}
\begin{proof}[Proof of Theorem \ref{theorem9.1''} and Corollary \ref{corollary1.3}]
By Corollary \ref{corollary8.1}, the given inequalities hold true and if equality holds true then the approximation constant is computed on an irreducible curve $B$ defined over $\KK$ and passing through $x$.   On the other hand, since $\epsilon_x(L) < \infty$ it follows that $\alpha_x(L)$ must be finite too and so $B$ must be rational by Theorem \ref{theorem9.1'}.
\end{proof}

\providecommand{\bysame}{\leavevmode\hbox to3em{\hrulefill}\thinspace}
\providecommand{\MR}{\relax\ifhmode\unskip\space\fi MR }
\providecommand{\MRhref}[2]{%
  \href{http://www.ams.org/mathscinet-getitem?mr=#1}{#2}
}
\providecommand{\href}[2]{#2}

\end{document}